\documentclass[12pt,reqno,a4paper]{amsart}
\usepackage{amsfonts,hyperref,amssymb}
\usepackage[dvipsnames]{xcolor}
\usepackage{lineno}
\usepackage[margin=2.5cm]{geometry}

\usepackage{hyperref}
\hypersetup{
    colorlinks=true,
    linkcolor=blue,
    citecolor=red,
}

\newcommand{\veps}{\varepsilon}
\newcommand{\R}{\mathbb{R}}

\newcommand{\C}{\mathbb{C}}

\newcommand{\A}{\mathcal{A}}

\newcommand{\D}{\mathbb{D}}

\newcommand{\uld}{\overline{\operatorname{logdens}}}
\newtheorem{theorem}{Theorem}[section]

\theoremstyle{plain}

{\theoremstyle{definition}
\newtheorem{example}{Example}[section]}

\newtheorem{lemma}{Lemma}[section]

{\theoremstyle{definition}
\newtheorem{remark}{Remark}[section]}

\numberwithin{equation}{section}

{\theoremstyle{definition}
}

\title[On the number of linearly independent rapid solutions]{On the number of linearly independent rapid solutions to linear differential and
linear difference equations}
\author[Heittokangas]{Janne Heittokangas}
\author[Yu]{Hui Yu}
\author[Zemirni]{M.~Amine Zemirni}
\address[Heittokangas\\ Yu\\ Zemirni]{University of eastern Finland, Department of Physics and Mathematics, P.O.~Box 111, 80101 Joensuu, Finland}

\email{janne.heittokangas@uef.fi \\ huiy@uef.fi \\ amine.zemirni@uef.fi}

\date{\today}

\begin{document}

\maketitle

\begin{abstract}
Assuming that $A_0,\ldots,A_{n-1}$ are entire functions and that $p\in \{0,\ldots,n-1\}$ is the smallest index such that $A_p$ is transcendental, then, by a classical theorem of Frei, each solution base of the differential equation
$f^{(n)}+A_{n-1}f^{(n-1)}+\cdots +A_{1}f'+A_{0}f=0$ 
contains at least $n-p$ entire functions of infinite order. Here, the transcendental coefficient $A_p$ dominates the growth of the polynomial coefficients  $A_{p+1},\ldots,A_{n-1}$. By expressing the dominance of $A_p$ in different ways, and allowing the coefficients $A_{p+1},\ldots,A_{n-1}$ to be transcendental, we show that the conclusion of Frei's theorem still holds along with an additional estimation on the asymptotic lower bound for the growth of solutions. At times these new refined results give a larger number of linearly independent solutions of infinite order than the original theorem of Frei. For such solutions, we show that $0$ is the only possible finite deficient value. Previously this property has been known to hold for so-called admissible solutions and is commonly cited as Wittich's theorem. Analogous results are discussed for linear differential equations in the unit disc, as well as for complex difference and complex $q$-difference equations.


\end{abstract}


\section{Introduction}\label{plane-sec} 

If the coefficients $A_0(\not\equiv 0),\ldots,A_{n-1}$ are analytic in a simply connected domain $D\subset \C$, then the differential equation
	\begin{equation}\label{lden}
	f^{(n)}+A_{n-1}f^{(n-1)}+\cdots +A_{1}f'+A_{0}f=0
	\end{equation}
possesses $n$ linearly independent analytic solutions in $D$.
In particular, if all the coefficients $A_0,\ldots,A_{n-1}$ are polynomials, then it is known that all non-trivial solutions $f$ of \eqref{lden} are entire functions of finite order and of regular growth, which implies $\log T(r,f) \asymp \log r$. In the case that the coefficients $A_0,\ldots,A_{n-1}$ are entire and at least one of them is transcendental, it follows that there exists at least one solution of \eqref{lden} of infinite order. This is a consequence of the following result due to M.~Frei, which can be considered as one of the
seminal results regarding the growth of solutions of \eqref{lden}.

\bigskip\noindent
\textbf{Frei's theorem.} (\cite[p.~207]{Frei0}, \cite[p.~60]{Laine})
\emph{Suppose that the coefficients in \eqref{lden} are entire, and that at least one of them is transcendental. Suppose that 
$p\in\{0,\ldots,n-1\}$ is the smallest index such that $A_p$ is transcendental, that is, the coefficients
$A_{p+1},\ldots,A_{n-1}$, if applicable, are polynomials. Then every solution base of \eqref{lden} has at
least $n-p$ solutions of infinite order.}
\bigskip

The following well-known result of H.~Wittich is one of the cornerstones of complex oscillation theory.
The original statement is for rational coefficients, but an easy modification
of the proof generalizes the result to small meromorphic coefficients.

\bigskip
\noindent
\textbf{Wittich's theorem.} (\cite[p.~62]{Laine}, \cite[p.~54]{Wittich})
\emph{Suppose that a meromorphic solution $f$ of \eqref{lden} is admissible in the sense that
	\begin{equation}\label{admissibl}
	T(r,A_j)=o(T(r,f)),\quad r\not\in E,\ j=0,\ldots,n-1,
	\end{equation}
where $E\subset[0,\infty)$ is a set of finite linear measure. Then $0$ is the only possible finite 
Nevanlinna deficient value for $f$.}

\bigskip
Recall that the Nevanlinna deficiency $\delta(a,f)$ for the $a$-points of a meromorphic 
function $f$ is defined by 
	$$
	\delta(a,f) = \liminf_{r\to\infty} \frac{m(r,a,f)}{T(r,f)}
	=1-\limsup_{r\to\infty}\frac{N(r,a,f)}{T(r,f)},\quad a\in\widehat{\C}.
	$$
If $\delta(a, f) > 0$, then $a$ is called a {\it Nevanlinna deficient value} of $f$.

As observed in \cite[p.~246]{HX}, the functions $f_1(z)=\exp(e^z)$ and $f_2(z)=z\exp(e^z)$ are linearly independent solutions of
	\begin{equation}\label{ex-frei}
	f''-(2e^z+1)f'+e^{2z}f=0.
	\end{equation}
Therefore all non-trivial solutions of \eqref{ex-frei} are of infinite order and admissible in the sense of \eqref{admissibl}. In contrast, according to Frei's theorem, the equation \eqref{ex-frei} has at least one solution of infinite order. Meanwhile, Wittich's theorem does not say anything about the number of linearly independent admissible solutions. This motivates us to find improvements of Frei's theorem, which will also address the number of linearly independent admissible solutions.

The key idea in Frei's theorem is that the transcendental coefficient $A_p$ dominates the growth of the polynomial coefficients $A_{p+1}, \ldots, A_{n-1}$. In the main results of this paper, we introduce different ways to express that the transcendental coefficient~$A_p$ dominates the growth of the coefficients $A_{p+1}, \ldots, A_{n-1}$, which are not necessarily polynomials. As a part of the conclusions, we obtain that the equation \eqref{lden} has at least $n-p$ linearly independent solutions $f$ which are admissible and, moreover, superior to the growth of the coefficient $A_p$ in the sense that
	\begin{equation}\label{superior}
	T(r,A_p) \lesssim \log T(r,f).
	\end{equation}
Solutions $f$ of \eqref{lden} satisfying \eqref{superior} are considered as rapid solutions.
Since any transcendental entire function $g$ satisfies 
	\begin{equation}\label{T}
	\liminf_{r\to\infty} \frac{T(r,g)}{\log r} = \infty,
	\end{equation}
see \cite[Theorem~1.5]{YY}, we deduce that the linearly independent solutions $f$ satisfying~\eqref{superior} are of infinite order. Thus Frei's theorem follows as a special case.

Regarding the differential equation \eqref{lden} in the unit disc $\D$,  the \textit{Korenblum space} $\A^{-\infty}=\cup_{q\geq 0}\A^{-q}$ introduced in \cite{K} takes the role of the polynomials. Here, $\A^{-q}$ for $q\in[0,\infty)$ is the growth space consisting of functions $f$ analytic in $\D$ and satisfying
	\begin{equation*}
	\underset{z\in \D}{\sup}(1-|z|^2)^q|f(z)|<\infty.
	\end{equation*}
On one hand, if all the coefficients $A_0,\ldots, A_{n-1}$ belong to $\A^{-\infty}$, then all non-trivial solutions of \eqref{lden} are of finite order of growth, see \cite[p.~36]{JH}. On the other hand, if at least one of the coefficients $A_0,\ldots, A_{n-1}$ does not belong to $\A^{-\infty}$, then \eqref{lden} possesses at least one solution of infinite order. This is a consequence of the following unit disc counterpart of Frei's theorem.

\bigskip\noindent
\textbf{First formulation of Frei's theorem in $\D$.} (\cite[Theorem 6.3]{JH})
\emph{Suppose that the coefficients $A_0,\ldots, A_{n-1}$ in \eqref{lden} are analytic in $\D$, and
that at least one of them is not in $A^{-\infty}$. Suppose that 
$p\in\{0,\ldots,n-1\}$ is the smallest index such that $A_p$ is  not in $A^{-\infty}$, that is, the coefficients $A_{p+1},\ldots, A_{n-1}$, if applicable, are in $\A^{-\infty}$. Then every solution base of \eqref{lden} has at least $n-p$ solutions of infinite order.}
\bigskip

Recall that a function $g$ meromorphic in $\D$ is called \textit{admissible} if  
	\begin{equation}\label{admissible-0}
	\limsup_{r\to 1^-}\frac{T(r,g)}{-\log (1-r)}=\infty,
	\end{equation}
otherwise $g$ is called \textit{non-admissible}.  As observed in \cite[p.~449]{HW}, for an admissible $g$ there exists a set $F\subset[0,1)$ with $\int_F \frac{dt}{1-t}=\infty$ such that 
	\begin{equation*}\label{admissible-1}
	\lim_{\substack{r\to 1^- \\ r\in F}}\frac{T(r,g)}{-\log (1-r)}=\infty.
	\end{equation*}
This is a unit disc analogue of \eqref{T}.
It is clear that the functions in $\A^{-\infty}$
are non-admissible. Conversely, the function $f(z)=\exp\left(\frac{1+z}{1-z}\right)$ has bounded
characteristic and hence it is non-admissible, but clearly $f\not\in \A^{-\infty}$. This gives
raise to the following second formulation of Frei's theorem in $\D$, which does not seem to appear
in the literature, but which follows easily from more general results in  Section~\ref{main-disc}. 

\bigskip\noindent
\textbf{Second formulation of Frei's theorem in $\D$.} 
\emph{Suppose that the coefficients $A_0,\ldots, A_{n-1}$ in \eqref{lden} are analytic in $\D$, and
that at least one of them is admissible. Suppose that 
$p\in\{0,\ldots,n-1\}$ is the smallest index such that $A_p$   is admissible,
that is, the coefficients $A_{p+1},\ldots, A_{n-1}$, if applicable, are non-admissible. 
Then every solution base of \eqref{lden} has at least $n-p$ solutions of infinite order.}

\bigskip

As observed in \cite[Example 1.4]{HKR2}, for $\beta>1,$ the functions $f_1(z)=\exp(\exp((1-z)^{-\beta}))$ and $f_2(z)=\exp((1-z)^{-\beta})\exp(\exp((1-z)^{-\beta}))$ are linearly independent infinite order solutions of 
	\begin{equation}\label{example-u-p}
	f''+A_1 f'+A_0 f=0,
	\end{equation}
where 
	$$
	A_1(z)=-\frac{2\beta\exp((1-z)^{-\beta})}{(1-z)^{\beta+1}}
	-\frac{\beta}{(1-z)^{\beta+1}}-\frac{1+\beta}{1-z}
	$$
and
	$$
	A_0(z)=\frac{\beta^2\exp(2(1-z)^{-\beta})}{(1-z)^{2\beta+2}}.
	$$
If $h_1(z)=\exp((1-z)^{-\beta})$ and $h_2(z)=(1-z)^{-(\beta+1)}$, then
	$$
	T(r,h_1)\asymp \int_0^{2\pi}\frac{d\theta}{|1-re^{i\theta}|^\beta}\asymp \frac{1}{(1-r)^{\beta-1}}
	\quad\textnormal{and}\quad T(r,h_2)=O(1).
	$$
Thus $A_0, A_1 \notin \A^{-\infty}$ are admissible and satisfy $T(r,A_0)=2T(r,A_1)+O(1)$. According to either formulation of Frei's theorem in $\D$, the equation \eqref{example-u-p} has at least one solution of infinite order.  Since all solutions of \eqref{example-u-p} are of infinite order, this leads us to consider possible improvements of Frei's theorems in $\D$.

The Nevanlinna deficiency for the $a$-points of a meromorphic function $f$ in $\D$ is defined
analogously as in the plane case simply by replacing ''$r\to\infty$'' with ''$r\to 1^-$''.
Differing from the plane case, we need to assume that $T(r,f)$ is unbounded. The unit disc
analogue of Wittich's theorem follows trivially by assuming that the set $E\subset [0,1)$ in \eqref{admissibl} now satisfies $\int_E\frac{dr}{1-r}<\infty$. 
The question on the number of linearly independent admissible solutions in Wittich's theorem is also valid in the unit disc. Note that the term ''admissible'' is used in two different meanings in the unit disc.

Slightly differing from the analogous situation in $\C$, the following two types of solutions of \eqref{lden} with coefficients analytic in $\D$ are considered as rapid solutions: 
\begin{itemize}
\item[(I)] Solutions $f$ satisfying \eqref{superior}, where $A_p$ is admissible.
\item[(II)] Solutions $f$ satisfying
	$$
	\log T(r,f) \gtrsim \log \int_{D(0,r)} |A_p(z)|^{\frac{1}{n-p}} dm(z),
	$$
where  $dm(z)$ is Lebesgue measure in the disc $D(0,r)$ and 
	$$
	\limsup_{r\to 1^-} \frac{\log \displaystyle \int_{D(0,r)} |A_p(z)|^{\frac{1}{n-p}} dm(z)}{- \log (1-r)} = \infty.
	$$
\end{itemize}
Here $A_p$ dominates the coefficients $A_{p+1}, \ldots, A_{n-1}$ in a certain way.

This paper is organized as follows. 
In Sections~\ref{main-plane} and \ref{main-disc}, the main results are stated in the cases of complex plane and the unit disc, and their sharpness is discussed in terms of examples.
A refinement of the standard order reduction method, needed for proving the main results, is given in Section~\ref{sec-reduction}.  The actual proofs are given in Sections~\ref{proof-plane} and \ref{proof-disc}. The analogous situation for linear difference  and $q$-difference equations is discussed in Sections~\ref{difference-sec} and \ref{q-difference-sec}, respectively.


\section{Results in the complex plane} \label{main-plane}
A refinement of Frei's theorem is given in \cite[Theorem~5.6]{HX} but is stated in terms of the number of linearly independent ``slow'' solutions $f$ of \eqref{lden} satisfying 
	$$
	\log T(r,f) = o(T(r,A_p)), \quad r\to\infty, \ r\notin E,
	$$
where $A_p$ dominates the growth of the coefficients $A_{p+1}, \ldots, A_{n-1}$ in a certain sense and $E\subset [0,\infty)$ is a set of finite linear measure.
However, the next example illustrates that some solutions may grow significantly slower than any of the coefficients.

\begin{example}
Let $\{z_n\}$ be a sequence defined by $z_{2n-1}=2^n$ and $z_{2n}=2^n+\varepsilon_n$, where
the numbers $\varepsilon_n>0$ are small, say
	$$
	0<\varepsilon_n<\exp\left(-\exp(2^n)\right),\quad n\geq 1.
	$$
Then \cite[Example~6]{GHW} shows that the canonical product
	$$
	f(z)=\prod_{n=1}^\infty\left(1-\frac{z}{z_n}\right)
	$$
is an entire solution of a differential equation
	\begin{equation}\label{2ode}
	f''+A_1 f'+A_0 f=0,
	\end{equation}
where the coefficients $A_1$ and $A_0$ are entire functions of infinite order of growth. Further restrictions
on the numbers $\varepsilon_n>0$ will induce even faster growth for $A_1$ and $A_0$. Meanwhile,
it is easy to see that $n(r,1/f)\asymp \log r$. Using
	$$
	r\int_1^\infty\frac{\log t}{t(r+t)}\, dt \leq 
	\int_1^r\frac{\log t}{t}\, dt+r\int_r^\infty \frac{\log t}{t^2}\, dt=O\left(\log^2 r\right)
	$$
together with (2.6.9) in \cite{Boas}, it follows that $T(r,f) \le \log M(r,f)=O\left(\log^2 r\right)$. 
\end{example}

The proof of Theorem~5.6 in \cite[p.~244]{HX} does not seem to support the exact formulation of \cite[Theorem~5.6]{HX}
because the set $I$ appearing in (5.1.31) is not in general the same as the set $I$ appearing
in (5.1.32). If these two sets are indeed different, then the set in (5.1.32) may affect on the validity of the $\limsup$ in (5.1.31). 

For reasons discussed above, we reformulate \cite[Theorem~5.6]{HX} such that it concerns the number of linearly independent rapid solutions, see Theorem~\ref{thhh1}. Moreover, the upper bound in \eqref{3'} is new.

\begin{theorem}\label{thhh1}
Let the coefficients $A_{0},\ldots,A_{n-1}$ in \eqref{lden} be entire functions such that at least one of them is transcendental.
Suppose that 
$p\in\{0,\ldots,n-1\}$ is the smallest index such that
	\begin{equation}\label{2}
	\limsup_{\substack{r\to \infty}}\sum_{j=p+1}^{n-1}\frac{T(r,A_{j})}{T(r,A_{p})}<1.
	\end{equation}
Then $A_p$ is transcendental, and every solution base of \eqref{lden} has at least $n-p$ rapid solutions $f$ for which 
	\begin{equation}\label{3'}
	T(r,A_{p}) \lesssim \log T(r,f) \lesssim \frac{R+r}{R-r}T(R,A_p),\quad r \not\in E,
	\end{equation}
where $E \subset [0,\infty)$ has finite linear measure, and $r<R<\infty$. For these solutions, the value $0$ is the only possible finite deficient value.
\end{theorem}

When $p= n-1$, the sum in \eqref{2} will be considered as zero, and the same situation applies in the next statements.

The upper bound of $\log T(r,f)$ in \eqref{3'} cannot be reduced to $ T(r,A_p)$, as is shown in Example~\ref{example01}(i) below.  Comparing \eqref{3'} with the classical inequalities
	$$
	T(r,A_p) \le \log M(r,A_p) \le \frac{R+r}{R-r} T(R,A_p),\quad r<R<\infty,
	$$
the quantities $\log T(r,f)$ and $\log M(r,A_p)$ seem to be comparable. 
Indeed, this is the case in the following result, but under a slightly different assumption.


\begin{theorem}\label{asymptotic-cor}
Let the coefficients $A_{0},\ldots,A_{n-1}$ in \eqref{lden} be entire functions such that at least one of them is transcendental. Suppose that $p\in\{0,\ldots,n-1\}$ is the smallest index such that
	\begin{equation}\label{2LM2}
	\limsup_{\substack{r\to \infty }}\sum_{j=p+1}^{n-1}\frac{\log^+ M(r,A_{j})}{\log^+ M(r,A_{p})}<1.
	\end{equation}
Then $A_p$ is transcendental, and every solution base of \eqref{lden} has at least $n-p$ rapid solutions $f$ for which
 	\begin{equation}\label{result-th1.3}
 	\log T(r,f)\asymp\log M(r,A_p),\quad r\not\in E,
 	\end{equation}
where $E \subset [0,\infty)$ has finite linear measure. For these solutions, the value $0$ is the only possible finite deficient value.
\end{theorem}

Conclusions \eqref{3'} and \eqref{result-th1.3} both imply \eqref{superior}, and therefore Frei's theorem is a particular case of Theorems~\ref{thhh1} and \ref{asymptotic-cor}. At times,  Theorems~\ref{thhh1} and \ref{asymptotic-cor} give a larger number of linearly independent solutions of infinite order than Frei's theorem. Indeed, the transcendental coefficients $A_0(z)=e^{2z}$ and $A_1(z)=-(2e^z+1)$ in \eqref{ex-frei} satisfy \eqref{2} and \eqref{2LM2} for $p=0$,
and the $\limsup$ in \eqref{2} or in \eqref{2LM2} is equal to $1/2$.

The following examples show that neither of Theorems~\ref{thhh1} and  \ref{asymptotic-cor} implies the other in the cases when the coefficients are of finite hyper-order or of finite order.

\begin{example}\label{example01}
(i) Let $A_1(z) = e^{e^z}$, and let $A_0$ be an entire function satisfying
	$$
	T(r,A_0) \sim \log M(r,A_0) \sim 2 T(r,A_1), \quad r\to \infty.
	$$
Such a function $A_0 $ exists by Clunie's theorem \cite{C}. Moreover, 
			$$
			T(r,A_1) \asymp \frac{e^r}{\sqrt{r}} \quad\text{and} \quad \log M(r,A_1)=e^r,
			$$
see \cite[p.~7]{H}. 
Therefore,
	$$
	\limsup_{r\to \infty} \frac{T(r,A_1)}{T(r,A_0)} = \frac{1}{2} < 1.
	$$
By Theorem~\ref{thhh1}, every non-trivial solution $f$ of \eqref{2ode} satisfies
	$$
	\frac{e^r}{\sqrt{r}} \lesssim \log T(r,f).
	$$
However, the asymptotic inequality $ \log T(r,f) \lesssim \frac{e^r}{\sqrt{r}}$ does not hold for all solutions $f$. In fact, we have
	$$
	\limsup_{r\to\infty} \frac{\log M(r,A_1)}{\log M(r,A_0)}= \limsup_{r\to\infty} \frac{\log M(r,A_1)}{2 T(r,A_1)} = \infty.
	$$ 
Thus, by Theorem~\ref{asymptotic-cor}, every solution base of \eqref{2ode} has at least one solution $f_0$ satisfying $\log T(r,f_0) \asymp \log M(r,A_1) = e^r$. In particular, Theorem~\ref{thhh1} is stronger than Theorem~\ref{asymptotic-cor}
in the sense that the number of rapid solutions given by Theorem~\ref{thhh1} is larger than that given by Theorem~\ref{asymptotic-cor}.

 (ii) Now, let $A_0(z)=e^{e^z}$, and let $A_1(z)$ be an entire function satisfying
	$$
	T(r,A_1) \sim \log M(r,A_1) \sim T(r,A_0)\asymp \frac{1}{\sqrt{r}} \log M(r,A_0), \quad r\to\infty.
	$$
Therefore,
	$$
	\limsup_{r\to\infty} \frac{T(r,A_1)}{T(r,A_0)}=1
	$$
and
	$$
	\limsup_{r\to\infty} \frac{\log M(r,A_1)}{\log M(r,A_0)} = \limsup_{r\to\infty} \frac{T(r,A_0)}{\log M(r,A_0)} = 0 <1.
	$$
Thus, Theorem~\ref{asymptotic-cor} is stronger than Theorem~\ref{thhh1} in this case.
\end{example}

\begin{example}
\noindent (i) Let $A_0(z)=E_{1/\varrho}(z)$ be Mittag-Leffler's function of order $\varrho>1/2$. We have, by \cite[p.~19]{H},
	$$
	T(r,A_0) \sim  \frac{1}{\pi \varrho} \log M(r,A_0) \sim \frac{1}{\pi \varrho} r^\varrho, \quad r\to\infty.
	$$
From \cite{C}, there exists an entire function $A_1(z)$ satisfying
	$$
	T(r,A_1) \sim \log M(r,A_1) \sim T(r,A_0), \quad r\to \infty.
	$$
Therefore,
	$$
	\limsup_{r\to\infty}\frac{T(r,A_1)}{T(r,A_0)} = 1
	$$
and
	$$
	\limsup_{r\to\infty}\frac{\log M(r,A_1)}{\log M(r,A_0)} = \limsup_{r\to\infty} \frac{T(r,A_0)}{\log M(r,A_0)} = \frac{1}{\pi \varrho}<1.
	$$
Thus, Theorem~\ref{asymptotic-cor} is stronger than Theorem~\ref{thhh1} in this case.

(ii) Now, let $A_1(z) = E_{1/\varrho}(z)$, hence
	$$
	T(r,A_1) \sim  \frac{1}{\pi \varrho} \log M(r,A_1), \quad r\to\infty,
	$$
and let $A_0(z)$ be an entire function satisfying
	$$
	T(r,A_0) \sim \log M(r,A_0) \sim \pi \varrho T(r,A_1), \quad r\to \infty.
	$$
Therefore,
	$$
	\limsup_{r\to\infty}\frac{T(r,A_1)}{T(r,A_0)} = \frac{1}{\pi \varrho} < 1
	$$
and 
	$$
	\limsup_{r\to\infty}\frac{\log M(r,A_1)}{\log M(r,A_0)} = \limsup_{r\to\infty} \frac{\log M(r,A_1)}{\pi \varrho T(r,A_1)} = 1.
	$$
Thus, Theorem~\ref{thhh1} is stronger than Theorem~\ref{asymptotic-cor} in this case. 
\end{example}

Sometimes, we can detect the number of rapid solutions when one coefficient dominates the rest of the coefficients along a curve. To this end, let $g$ be an entire function, and let $\mathcal{M}_g := \{z \in \C : |g(z)|=M(|z|,g)\}$.
For example, if $g(z)=z$ then $\mathcal{M}_g = \C$, while if $g(z)= e^z$ then $\mathcal{M}_g= \R_+$.
 For any entire $g$, the set $\mathcal{M}_g$ contains at least one curve tending to infinity, although isolated points in $\mathcal{M}_g$ are also possible. Any curve in $\mathcal{M}_g$ tending to infinity is called a {\it maximum curve} for $g$. For more details, see \cite{T}.

\begin{theorem}\label{sdsd}
Let the coefficients $A_{0},\ldots,A_{n-1}$ in \eqref{lden} be entire functions. Suppose that
$p\in\{0,\ldots,n-1\}$ is the smallest index such  that $A_p$ is transcendental and 
	\begin{equation}\label{p-u-1}
	\limsup _{\substack{z \rightarrow \infty\\ z \in \Gamma}} \sum_{j=p+1}^{n-1} \frac{1}{\eta_j}\frac{\left|A_{j}(z)\right|^{\eta_j}}{\left|A_{p}(z)\right|}<1
	\end{equation}
holds for some constants $\eta_j>1$, where $\Gamma$ is a maximum curve for $A_p$. Then every solution base of \eqref{lden} has at least $n-p$ rapid solutions $f$ for which
	$$
	\log T(r,f) \gtrsim \log M(r,A_{p}),\quad r\not\in E,
	$$
where $E \subset [0,\infty)$ has finite linear measure.
\end{theorem}

The condition \eqref{p-u-1} in Theorem~\ref{sdsd} does not restrict the growth of the coefficients globally, and therefore \eqref{p-u-1} does not imply the admissibility of the rapid solutions. 

Differing from the analogous situation in Theorem~\ref{asymptotic-cor}, the next example shows that the asymptotic comparability between $\log T(r,f)$ and $\log M(r,A_p)$ does not always occur in the conclusion of Theorem~\ref{sdsd}.
\begin{example}\label{ex2.4}
Consider the differential equation
	$$
	f'' + e^{-z^2} f' + e^z f=0.
	$$
Condition \eqref{p-u-1} clearly holds for $p=0$  along the positive real axis, which is the maximum curve for $e^z$. Thus, all non-trivial solutions $f$ satisfy
	$$
	\log T(r,f) \gtrsim \log M(r,e^z) = r.
	$$
However, the asymptotic inequality $\log T(r,f) \lesssim \log M(r,e^z)$ doesn't hold for all solutions. Indeed, according to Theorem~\ref{asymptotic-cor}, the equation above has at least one solution $f_0$ satisfying
	$$
	\log T(r,f_0) \asymp \log M(r,e^{-z^2}) =r^2.
	$$

\end{example}

The following example shows  that in some cases the number of linearly independent rapid solutions, given by Theorem~\ref{sdsd}, is larger than the number given by Theorem~\ref{thhh1} or Theorem~\ref{asymptotic-cor}.

\begin{example}
Consider the differential equation
	$$
	f'' + e^{-z} f' + e^z f=0.
	$$
Theorems~\ref{thhh1} and \ref{asymptotic-cor} both assert that each solution base contains at least one solution $f$ satisfying $\log T(r,f) \gtrsim r$. In contrast, the condition \eqref{p-u-1} holds for $p=0$ along the positive real axis. Thus, all non-trivial solutions $f$ satisfy $\log T(r,f) \gtrsim r$.
\end{example}

Finally, we give an example to show that our main results are refinements of Frei's theorem in the sense that the asymptotic comparability can sometimes be used to find more solutions of infinite order.

\begin{example}
The function $f_1(z) = e^{e^z}$ is an infinite order solution of the equation
	\begin{equation}\label{eq2}
	f'' + (e^{z^2}-e^z)f' - (e^{z^2+z}+e^z) f =0.
	\end{equation}
Let $f_2$ be any solution of \eqref{eq2} linearly independent to $f_1$.
Frei's theorem cannot be used to conclude that $f_2$ is of infinite order.
However, according to any of Theorems~\ref{thhh1}, \ref{asymptotic-cor} or \ref{sdsd}, $f_2$ must satisfy $\log T(r,f_2) \gtrsim T(r,e^{z^2}) \asymp r^2$. Meanwhile, $\log T(r,f_1) \asymp r$.    
\end{example}

\section{Results in the unit disc}\label{main-disc}
The next result is a unit disc counterpart of Theorem~\ref{thhh1}.

\begin{theorem}\label{th.1-ud-1}
Let the coefficients $A_0,\ldots,$ $A_{n-1}$ in \eqref{lden} be analytic functions in $\D$ such that at least one of them is admissible. Suppose that $p\in\{0,\ldots,n-1\}$ is the smallest index such that 
	\begin{equation}\label{i-u-2}
	\limsup_{r\to 1^-}\sum_{j=p+1}^{n-1}\frac{T(r,A_{j})}{T(r,A_{p})}<1.
	\end{equation}
Then $A_p$ is admissible, and every solution base of \eqref{lden} has at least $n-p$ rapid solutions $f$ for which
	\begin{equation}\label{3'-unit}
	T(r,A_{p}) \lesssim \log T(r,f) \lesssim \frac{R+r}{R-r} T(R,A_p), \quad r \not\in E,
	\end{equation}
where $E\subset[0,1)$ is a set with $\int_{E}\frac{dr}{1-r}<\infty$, and $0<r<R<1$.  For these solutions, the value $0$ is the only possible finite deficient value.
\end{theorem}

Analogously to the case of the complex plane, the asymptotic comparability between $\log T(r,f)$ and $ \log M(r,A_p)$ is considerable in the unit disc as well. Indeed, the unit disc counterpart of Theorem~\ref{asymptotic-cor} is given as follows.

\begin{theorem}\label{asymptotic-cor-u1}
Let the coefficients $A_0,\ldots,$ $A_{n-1}$ in \eqref{lden} be analytic in $\D$. Suppose that $p\in\{0,\ldots,n-1\}$ is the smallest index such that $A_p$ is admissible and
	\begin{equation}\label{2LM-u1}
	\limsup_{r\to 1^-}\sum_{j=p+1}^{n-1}\frac{\log^+ M(r,A_{j})}{\log^+ M(r,A_{p})}<1.
	\end{equation}
Then every solution base of \eqref{lden} has at least $n-p$ rapid solutions $f$ for which 
	\begin{equation}\label{conclusion2-disc}
	\log T(r,f) \asymp \log M(r,A_p),\quad r\not\in E,
	\end{equation}
where $E\subset[0,1)$ is a set with $\int_{E}\frac{dr}{1-r}<\infty$. For these solutions, the value $0$ is the only possible finite deficient value.
\end{theorem}

From \eqref{3'-unit} or \eqref{conclusion2-disc}, using \eqref{admissible-0}, we easily get that there are at least $n-p$ linearly independent solutions of infinite order. Thus the second formulation of Frei's theorem is a particular case of Theorems~\ref{th.1-ud-1} and  \ref{asymptotic-cor-u1}.

At times Theorems~\ref{th.1-ud-1} and \ref{asymptotic-cor-u1} give a larger number of linearly independent solutions of  infinite order than the second formulation of Frei's theorem. Indeed, the admissible coefficients $A_1(z)$ and $A_0(z)$ in \eqref{example-u-p} satisfy \eqref{i-u-2} and \eqref{2LM-u1} for $p=0$ and the $\limsup$ in both \eqref{i-u-2} and \eqref{2LM-u1} is equal to $1/2$.

The following example illustrates the differences between Theorems~\ref{th.1-ud-1} and \ref{asymptotic-cor-u1}, without restricting to any pre-given growth scale for the coefficients.

\begin{example}
(i) Let $\mu(r)$ and $\lambda(r)$ be two non-negative unbounded functions on the interval $[0,1)$ satisfying $\mu(r) = o(\lambda(r)),\; r\to 1^-$, and   the hypotheses in \cite[Theorem~I]{L}. Then there exist analytic functions $A_0$ and $A_1$ in $\D$ satisfying
	$$
	T(r,A_0) \sim  \mu(r) \sim 2T(r,A_1), \quad \log M(r,A_0) \sim \lambda(r) \sim \log M(r,A_1), \quad r\to 1^-.
	$$ 
Thus, by Theorem~\ref{th.1-ud-1}, all non-trivial solutions $f$ of 
	\begin{equation}\label{ex-disc}
	f'' + A_1  f' + A_0 f=0
	\end{equation}
satisfy $\log T(r,f) \gtrsim T(r,A_0) \sim \mu(r), \; r\to 1^-$. However, by Theorem~\ref{asymptotic-cor-u1}, there exists at least one solution $f_0$ of \eqref{ex-disc} satisfying $\log T(r,f_0) \asymp \log M(r,A_1) \sim \lambda(r), \; r\to 1^-$.
Since  $\mu(r) = o(\lambda(r)),\; r\to 1^-$, it follows that the upper bound of $\log T(r,f)$ in \eqref{3'-unit} cannot be reduced to $ T(r,A_0)$.

(ii) If we choose the analytic coefficients $A_0$ and $A_1$ in the following way
	$$
	T(r,A_0) \sim  \mu(r) \sim T(r,A_1), \quad \log M(r,A_0) \sim \lambda(r) \sim 2 \log M(r,A_1), \quad r\to 1^-,
	$$
then, by Theorem~\ref{th.1-ud-1}, each solution base of \eqref{ex-disc} has at least one solution $f$ satisfying $\log T(r,f) \gtrsim T(r,A_1) \sim \mu(r),\; r\to 1^-$. In contrast, Theorem~\ref{asymptotic-cor-u1} asserts that all non-trivial solutions $f$ of \eqref{ex-disc} satisfy $\log T(r,f) \asymp \log M(r,A_0) \sim \lambda(r), \; r\to 1^-$.
	
\end{example}

A maximum curve for an analytic function $g(z)$ in $\D$ is a curve emanating from the origin and tending to a point on $\partial \D$ and consists of points $z\in\D$ for which $\left|g(z)\right|=M\left(|z|, g\right)$.

\begin{theorem}\label{pointwise th}
Let the coefficients $A_0,\ldots,$ $A_{n-1}$ in \eqref{lden} be analytic functions in $\D$. Suppose that $p\in\{0,\ldots,n-1\}$ is the smallest index such that $A_p$ is admissible and
	\begin{equation*}\label{p-u-1-1}
\limsup _{\substack{z \rightarrow 1^-\\ z \in \Gamma}} \sum_{j=p+1}^{n-1} \frac{1}{\eta_j} \frac{\left|A_{j}(z)\right|^{\eta_j}}{\left|A_{p}(z)\right|}<1
	\end{equation*}
holds for some constants $\eta_j>1$, where $\Gamma$ is a maximum curve of $A_p$. Then every solution base of \eqref{lden} has at least $n-p$ rapid solutions $f$ for which 
	$$
	\log T(r, f) \gtrsim \log M\left(r, A_{p}\right)  ,\quad  r\not\in E,
	$$
where $E\subset[0,1)$ is a set with $\int_{E}\frac{dr}{1-r}<\infty$.
\end{theorem}

Similar to Example~\ref{ex2.4}, the following example shows that the comparability between $\log T(r,f)$ and $\log M(r,A_p)$ in Theorem~\ref{pointwise th} does not always occur.

\begin{example} Let $A_0$ and $A_1$ be admissible analytic functions in $\D$ defined by
	$$
	A_1(z) = \exp \left\{ \frac{-1}{(1-z)^{2\beta}}\right\}  \quad \text{and} \quad A_0(z)=\exp \left\{ \frac{1}{(1-z)^{\beta}}\right\}, \quad \beta>1.
	$$
Along the maximum curve for $A_0$, which is the line segment  $\Gamma = (0,1)$, we easily find 
	$$
	\limsup _{\substack{z \rightarrow 1^-\\ z \in \Gamma}}  \frac{1}{\eta}\frac{\left|A_{1}(z)\right|^{\eta}}{\left|A_{0}(z)\right|}=0,
	$$
for any $\eta>1$. Thus, according to Theorem~\ref{pointwise th}, all non-trivial solutions $f$ of \eqref{ex-disc} satisfy
	$$
	\log T(r,f) \gtrsim \log M(r,A_0) = \frac{1}{(1-r)^{\beta}}.
	$$
However, the asymptotic inequality $\log T(r,f) \lesssim \log M(r,A_0)$ does not hold for all solutions. Indeed, from Theorem~\ref{asymptotic-cor-u1}, there exist at least one solution $f_0$ satisfying 
	$$
	\log T(r,f_0) \asymp \log M(r,A_1) = \frac{1}{(1-r)^{2\beta}}.
	$$
\end{example}

Recall that the upper linear density of a set $E\subset[0,1)$ is given by
	$$
	\overline{d}(E) :=\limsup_{r\to 1^-}\frac{1}{1-r}{\int_{E \cap [r,1)} dr}.
	$$
It is clear  that $0\le \overline{d}(E) \le 1$ for any set $E \subset [0,1)$.
\begin{theorem}\label{s-int}
Let the coefficients $A_{0},\ldots,A_{n-1}$ in \eqref{lden} be analytic functions in $\D$. Suppose that $p\in\{0,\ldots,n-1\}$ is the smallest index such that  $A_p$ is admissible and
	\begin{equation}\label{limsup-int}
	\limsup_{r\to 1^-}  \sum_{j=p+1}^{n-1} \left( \frac{n-j}{n-p}\right) \frac{\displaystyle\int_0^{2\pi} |A_j(re^{i\theta})|^{\frac{1}{n-j}} d\theta}{\displaystyle\int_0^{2\pi} |A_p(re^{i\theta})|^{\frac{1}{n-p}} d\theta}<1.
	\end{equation} 
Then every solution base of \eqref{lden} has at least $n-p$ solutions $f$ for which
	\begin{equation*}
	\log T(r,f) \asymp \log {\displaystyle\int_0^{2\pi} |A_p(re^{i\theta})|^{\frac{1}{n-p}} d\theta}, \quad r\notin E,
	\end{equation*}
where $E\subset[0,1)$ is a set with $\overline{d}(E)<1$. These solutions are rapid in the sense of (I), and the value $0$ is their only possible finite deficient value.
\end{theorem}

The quantities
	\begin{equation}\label{p-mean}
	\int_0^{2\pi} |A_j(re^{i\theta})|^{\frac{1}{n-j}} d\theta
	\end{equation}
are used  to measure the growth of the coefficients $A_0, \ldots, A_{n-1}$  in results parallel to Theorem~\ref{s-int} in \cite{CHR}. Note that the assumption \eqref{limsup-int} is more delicate than the corresponding assumptions on the orders of growth in \cite{CHR}.

To see that the second formulation of Frei's theorem is a particular case of Theorem~\ref{s-int}, we first make use of Lemma~2 in \cite[p.~52]{JH2}, which allows us to avoid the exceptional set $E$ of density $\overline{d}(E)<1$. Second, we notice that if $A_p$ is admissible, then
	\begin{equation}\label{adm-sint}
	\limsup_{r\to 1^-} \frac{\log {\displaystyle\int_0^{2\pi} |A_p(re^{i\theta})|^{\frac{1}{n-p}} d\theta}}{- \log (1-r)} = \infty.
	\end{equation}
Indeed, it follows from Jensen's inequality that
	\begin{align}\label{r.solI}
	\log^+ {\int_0^{2\pi} |A_p(re^{i\theta})|^{\frac{1}{n-p}} d\theta} &\gtrsim m(r,A_p)=T(r,A_p).
	\end{align}
Therefore, making use of \eqref{admissible-0} yields \eqref{adm-sint}.

The previous results require the existence of at least one coefficient of \eqref{lden} being admissible, which works with the second formulation of Frei's theorem. In the following we require that at least one of the coefficients is not in $\A^{-\infty}$, which is more suitable for the first formulation of Frei's theorem.

\begin{theorem}\label{sdsd2-disc}
Let the coefficients $A_{0},\ldots,A_{n-1}$ in \eqref{lden} be analytic functions in $\D$ such that at least one of them does not belong to $\A^{-\infty}$. Suppose that $p\in\{0,\ldots,n-1\}$ is the smallest index such that 
	\begin{equation}\label{con1}
	\limsup_{r\to 1^-}  \sum_{j=p+1}^{n-1} \left( \frac{n-j}{n-p}\right) \frac{\displaystyle\int_{D(0,r)} \left| A_j(z) \right|^{\frac{1}{n-j}} dm(z)}{\displaystyle\int_{D(0,r)} \left| A_p(z) \right|^{\frac{1}{n-p}} dm(z)}<1.
	\end{equation} 
Then $A_p\notin \A^{-\infty}$, and every solution base of \eqref{lden} has at least $n-p$ solutions $f$ for which
	\begin{equation}\label{conclusion2}
	\log T(r,f) \asymp \log \int_{D(0,r)} \left| A_p(z) \right|^{\frac{1}{n-p}} dm(z), \quad r\notin E,
	\end{equation}
where $E\subset[0,1)$ is a set with $\int_{E}\frac{dr}{1-r}<\infty$. These solutions are rapid in the sense of (II), and the value $0$ is their only possible finite deficient value.
\end{theorem}

To see that the first formulation of Frei's theorem is a particular case of Theorem~\ref{sdsd2-disc}, it suffices to prove the following claim: Suppose that $g(z)$ is an analytic function in $\D$. Then  $g\notin\A^{-\infty}$ if and only if, for any $\kappa \in (0,1)$, 
	\begin{equation}\label{int-tran}
	\limsup_{r\to 1^-}\frac{\displaystyle \log^+\int_{D(0,r)}|g(z)|^\kappa dm(z)}{-\log (1-r)} = \infty.
	\end{equation}

To prove this claim, we modify \cite[Example~5.4]{HKR3}.
First, assume that $g \notin \A^{-\infty}$ and that \eqref{int-tran} does not hold, i.e., there exist $r_0 \in (0,1)$, $\kappa\in (0,1)$ and $C>0$ such that
	\begin{equation} \label{eq1}
	\int_{D(0,r)}|g(z)|^\kappa dm(z) \le  \frac{1}{(1-r)^C}, \quad r \in (r_0,1).
	\end{equation}
Using sub-harmonicity, we obtain
	$$
	|g(z)|^\kappa \leq \frac{1}{2 \pi} \int_{0}^{2 \pi}\left|g\left(z+t e^{i \theta}\right)\right| ^\kappa d \theta,  \quad 0<t< 1-|z|.
	$$
Multiplying both sides by $t$ and integrating from $0$ to $\frac{1-|z|}{2}$, it follows
	\begin{align*}
	\frac{1}{2}\left(\frac{1-|z|}{2}\right)^{2}|g(z)|^\kappa  \le \frac{1}{2\pi} \int_{D\left(0,\frac{1+|z|}{2}\right)} |g(\xi)|^\kappa dm(\xi).
	\end{align*}
Therefore, making use of \eqref{eq1} yields
	$$
	|g(z)| \lesssim \frac{1}{(1-|z|)^D}, \quad D=\frac{C+2}{\kappa},
	$$
which implies that $g\in \A^{-\infty}$ and this is a contradiction.  Conversely, if $g \in \A^{-\infty}$, then the $\limsup$ in \eqref{int-tran} is clearly finite.


\section{Lemmas on the order reduction method} \label{sec-reduction}

Lemma~\ref{5.3} below appears in \rm \cite[p.~234]{HX} but the proof is written in Chinese. 
A difference analogue of this lemma will be given as  Lemma~\ref{d5.3} below. The reader should have no problem in verifying Lemma~\ref{5.3} by studying the proof of Lemma~\ref{d5.3} and using the lemma on the logarithmic derivative instead of the lemma on the logarithmic 
difference. 

\begin{lemma}\label{5.3}
{\rm (\cite[Lemma~5.3]{HX})}
Suppose that $f_{0,1},\ldots,f_{0,n}$ are linearly independent 
meromorphic functions. Define inductively
	\begin{equation}\label{lin-indep-sols}
	f_{q,s}=\left(\frac{f_{q-1,s+1}}{f_{q-1,1}}\right)',\quad 1\leq q\leq n-1,\ 1\leq s\leq n-q.
	\end{equation}
Then
	\begin{equation}\label{01}
	T(r,f_{q,s})\lesssim \sum_{l=1}^{q+s} T(r,f_{0,l})+\log r,\quad r\not \in E_1,\nonumber
	\end{equation}
where $E_1\subset[0,\infty)$ is a set of finite linear measure. 
\end{lemma}

Using the the standard estimate for the logarithmic derivatives in the unit disc \cite[pp.~241--246]{Ne}, it is easy to obtain the following unit disc counterpart of Lemma~\ref{5.3}.  

\begin{lemma}\label{5.3-unit}
	Suppose that $f_{0,1},\ldots,f_{0,n}$ are linearly independent 
	meromorphic functions in $\D$. Define the functions $f_{q,s}$ as in \eqref{lin-indep-sols}.	Then
	\begin{equation}\label{01-unit}
	T(r,f_{q,s})\lesssim \sum_{l=1}^{q+s} T(r,f_{0,l})+\log \frac{1}{1-r},\quad r\not \in E_2,\nonumber
	\end{equation}
	where $E_2\subset[0,1)$ is a set with $\int_{E_2}\frac{dr}{1-r}<\infty$. 
\end{lemma}

A version of the following lemma is included in the proof of Theorem~5.6 in \cite[p.~244]{HX}. The precise form of the differential polynomials \eqref{CC} does not appear in \cite{HX}, but it is needed for proving Theorems~\ref{s-int} and ~\ref{sdsd2-disc}.

\begin{lemma}\label{coeff_Ap}
Let the coefficients $A_{0},\ldots,A_{n-1}$ in \eqref{lden} be meromorphic functions in a simply connected domain $D$, and let $f_{0,1},\ldots,f_{0,n}$ be linearly independent 
solutions of the equation \eqref{lden}. Define the functions $f_{q,s}$ as in \eqref{lin-indep-sols}. Then, for $p \in \{0, 1, \ldots, n-1\}$, we have
	\begin{equation}\label{4}
	-A_{p}= C_{n}+A_{n-1}C_{n-1}+\cdots+A_{p+1}C_{p+1},
	\end{equation}
where $C_{p+1}, \ldots, C_n$ have the following form
	\begin{equation} \label{CC}
	C_k = \sum_{{ l_0 + l_1 + \cdots + l_p = k-p}} K_{l_0, l_1, \ldots, l_p}
	\frac{f_{0,1}^{(l_0)}}{f_{0,1}}\, \frac{f_{1,1}^{(l_{1})}}{f_{1,1}}\,\cdots\,\frac{f_{p,1}^{(l_{p})}}{f_{p,1}}, \quad p+1 \le k \le n.
	\end{equation}
Here $0\leq l_0,l_1,\ldots,l_p\leq k-p$ and $K_{l_0, l_1, \ldots, l_p}$ are absolute positive constants.
\end{lemma}
\begin{proof}
We rename the coefficients $A_{0},\ldots,A_{n-1}$ by $A_{0,0},\ldots,A_{0,n-1}$.
Using the standard order reduction method as in \cite[p.~1233]{GSW} or in \cite[p.~60]{Laine}, we obtain, for a fixed $q\in \{1,\ldots,n-1\}$, that the functions $f_{q,s}$ in \eqref{lin-indep-sols} are linearly independent solutions of the equation
	\begin{equation} \label{qlde}
	f^{(n-q)} + A_{q, n-q-1} f ^{(n-q-1)} + \cdots + A_{q,0}f =0, 
	\end{equation}
where
	\begin{equation}\label{coeff}
	A_{q,j} = A_{q-1,j+1} + \sum_{k=j+2}^{n-q+1} {k\choose j+1} A_{q-1,k} \frac{ f_{q-1,1} ^{(k-j-1)}}{f_{q-1,1}},
	\quad j=0,\ldots,n-q-1.
	\end{equation}
In the case $q=p$, the function $f_{p,1}$ is a solution of \eqref{qlde}, and therefore 
	\begin{equation}\label{2.10}
	-A_{p,0}=\frac{f_{p,1}^{(n-p)}}{f_{p,1}}+A_{p, n-p-1}\frac{f_{p,1} ^{(n-p-1)}}{f_{p,1}} 
	+\cdots+A_{p,1}\frac{f_{p,1}'}{f_{p,1}}.
	\end{equation}	
We need to write the coefficients $A_{p,i}$ in \eqref{2.10} in terms of the coefficients $A_{0,0}, \ldots, A_{0,n-1}$. For that we prove by induction on $m = 1, \ldots, p$, that
	\begin{eqnarray}\label{claim-A}
	A_{p,i}&=&A_{p-m,i+m}+A_{p-m,i+m+1}C_{i,p-m,i+m+1}+\cdots +A_{p-m,n-p+m}C_{i,p-m,n-p+m},
	\end{eqnarray}
where $i=0,\ldots,n-p-1$,
	\begin{equation}\label{claim-C}
        C_{i,p-m,s+m}=\sum_{l_{p-m} + \cdots + l_{p-1} =s-i}K_{l_{p-m}, \ldots, l_{p-1}}\frac{f_{p-m,1}^{(l_{p-m})}}{f_{p-m,1}}\frac{f_{p-m+1,1}^{(l_{p-m+1})}}{f_{p-m+1,1}}\ldots\frac{f_{p-1,1}^{(l_{p-1})}}{f_{p-1,1}}, 
        \end{equation}
and $s= i+1, \ldots, n-p$.

When $m=1$, we get \eqref{claim-A} from \eqref{coeff} with
	\begin{equation*}\label{C--p-1}
	\begin{split}
C_{i,p-1,i+2}&={i+2\choose i+1}\frac{f_{p-1,1}^{(1)}}{f_{p-1,1}},\\
C_{i,p-1,i+3}&={i+3\choose i+1}\frac{f_{p-1,1}^{(2)}}{f_{p-1,1}},\\
 & \;\ \vdots\\
C_{i,p-1,n-p+1}&={n-p+1\choose i+1}\frac{f_{p-1,1}^{(n-p-i)}}{f_{p-1,1}}.
	\end{split}
	\end{equation*}
	
Now, we suppose that  \eqref{claim-A} and \eqref{claim-C} hold for $m$, and we aim to prove that they hold for $m+1$. Hence, by applying \eqref{coeff} into the coefficients  $A_{p-m,i+m}, A_{p-m,i+m+1} ,\ldots ,A_{p-m,n-p+m}$ in \eqref{claim-A}, and after rearranging the terms, we obtain
	\begin{eqnarray*}
	A_{p,i}&=&A_{p-m-1,i+m+1}+A_{p-m-1,i+m+2}C_{i,p-m-1,i+m+2}+\cdots +A_{p-m-1,n-p+m+1}C_{i,p-m-1,n-p+m+1},
	\end{eqnarray*}
where, for $j= i+1, \ldots, n-p$,
	\begin{equation}\label{eqeq}
	C_{i,p-m-1, j+m+1} = \sum_{s=i}^{j}  {j+m+1 \choose s+m+1} C_{i,p-m,s+m} \frac{f_{p-m-1,1}^{(j-s)}}{f_{p-m-1,1}},
	\end{equation}
and $C_{i,p-m,i+m} \equiv 1$.
By substituting \eqref{claim-C} into \eqref{eqeq}, we easily deduce 
	\begin{equation*}
        C_{i,p-m-1,j+m+1}=\sum_{l_{p-m-1} + \cdots + l_{p-1} =j-i}K_{l_{p-m-1}, \ldots, l_{p-1}}\frac{f_{p-m-1,1}^{(l_{p-m-1})}}{f_{p-m-1,1}} \cdots\frac{f_{p-1,1}^{(l_{p-1})}}{f_{p-1,1}}.
        \end{equation*}
Hence, we complete the proof of \eqref{claim-A} and \eqref{claim-C} for every $m = 1, \ldots, p$. In particular, when $m=p$, we obtain from \eqref{claim-A} that
	\begin{equation}\label{rearrangement}	
A_{p,i} =A_{0,p+i}+A_{0,p+i+1}C_{i,0,p+i+1}+A_{0,p+i+2}C_{i,0,p+i+2}+\cdots+A_{0,n}C_{i,0,n},
\end{equation}
where $0\leq i\leq  n-p-1$. 
Again, by substituting \eqref{rearrangement} into \eqref{2.10} for every $0\leq i\leq  n-p-1$, and by rearranging the terms, we get 
	\begin{equation*} 
	-A_{0,p	}=C_{n}+A_{0,n-1}C_{n-1}+\cdots+A_{0,p+1}C_{p+1},
	\end{equation*}
	where 
	\begin{equation}\label{Cc2}
	C_j = C_{0,0,j} + \sum_{k= n-j}^{n-p-1} C_{n-p-k,0, j} \frac{f_{p,1}^{(n-p-k)}}{f_{p,1}}.
	\end{equation}
Finally, from \eqref{Cc2} and  \eqref{claim-C}, we can easily get \eqref{CC}.
\end{proof}


\section{Proofs of the results in the complex plane} \label{proof-plane}

To prove Theorems \ref{asymptotic-cor} and \ref{sdsd}, we need the following version of the lemma on the logarithmic derivative, which differs from the standard versions in \cite{G} in the sense that the upper estimate involves an arbitrary $R\in (r,\infty)$ as opposed to a specifically chosen $R=\alpha r$, where $\alpha>1$. 

\begin{lemma}\label{G-estimate}
Let $0<R<\infty$, $\alpha>1$, and let $f$ be a meromorphic function in $\C$. Suppose that $k,j$ are integers with $k>j\geq 0$, and $f^{(j)}\not\equiv 0$. Then there exists a~set $E_3\subset[0,\infty)$ that has finite linear measure such that for all $z$ satisfying $|z|=  r \in (0,R) \setminus E_3$, we have 
	\begin{equation}\label{G-estimate-1}
	\left|\frac{f^{(k)}(z)}{f^{(j)}(z)}\right|
	\lesssim \left\{  \frac{R}{R-r}\left(1+\log ^{+} R+\log ^{+} \frac{1}{R-r}+T(R, f)\right)\right\}^{(1+\alpha)(k-j)}.
	\end{equation}
Moreover, if $k=1$ and $j=0$, then the logarithmic terms in \eqref{G-estimate-1} can be omitted.
\end{lemma}

\begin{proof}
Let $\{a_m\}$ denote the sequence of zeros and poles of $f^{(j)}$ listed according to multiplicity
and ordered by increasing modulus. Let $n(r)$ denote the number of points $a_m$ in $D(0,r)$,
and let $N(r)$ denote the corresponding integrated counting function.

Consider the case $k=1$ and $j=0$ first. By a standard reasoning
based on the Poisson-Jensen formula, we obtain
	\begin{eqnarray*}
	\left|\frac{f'(z)}{f(z)}\right|&\leq& \frac{\varrho}{(\varrho-r)^2}\int_0^{2\pi}
	|\log |f(\varrho e^{i\theta})||\, d\theta+\sum_{|a_m|<\varrho}\left(\frac{1}{|z-a_m|}
	+\frac{|a_m|}{|\varrho^2-\bar a_mz|}\right),
	\end{eqnarray*}
where $|z|=r<\varrho< R$. From the first fundamental theorem, it follows that
	$$
	\int_0^{2\pi}
	|\log |f(\varrho e^{i\theta})||\, d\theta\leq 4\pi (T(\varrho,f)+O(1)).
	$$
Clearly
	$$
	\sum_{|a_m|<\varrho}\frac{|a_m|}{|\varrho^2-\bar a_mz|}\leq \frac{n(\varrho)}{\varrho-r}.
	$$	
Let $U$ be the collection of discs $D(a_m,1/m^\alpha)$ if $a_m\neq 0$ and $D(a_m,1)$ if $a_m=0$. Then
the projection $E_3$ of $U$ onto $[0,\infty)$ has a linear measure at most
	$$
	1+\sum_{m=1}^\infty \frac{2}{m^\alpha}<\infty.
	$$
Let $L$ be the number of points $a_m$ at the origin. If $z\not\in U$, we have
	$$
	\sum_{|a_m|<\varrho}\frac{1}{|z-a_m|}\leq L+\sum_{|a_m|<\varrho}m^\alpha
	\leq L+\sum_{|a_m|<\varrho}n(\varrho)^\alpha\leq L+n(\varrho)^{1+\alpha}.
	$$
Since
	$$
	N(R)-N(\varrho)=\int_\varrho^R\frac{n(t)}{t}\, dt \geq n(\varrho)\frac{R-\varrho}{R},
	$$
it follows that
	\begin{equation} \label{n}
	n(\varrho)\leq \frac{RN(R)}{R-\varrho}\leq \frac{2RT(R,f)+O(R)}{R-\varrho}.
	\end{equation}
Choosing $\varrho=(R+r)/2$ and putting everything together, we deduce
	$$
	\left|\frac{f'(z)}{f(z)}\right|\lesssim 
	\left(\frac{R}{R-r}\right)^{1+\alpha}(T(R,f)+1)^{1+\alpha},\quad z\not\in U,
	$$
which implies the assertion in the case when $k=1$ and $j=0$.

Consider next the general case. Standard estimates yield
	\begin{equation}\label{derivative}
	T(s,f^{(m)})\lesssim 1+\log^+R+\log^+\frac{1}{R-r}+T(R,f),\quad s=\frac{r+R}{2}.
	\end{equation}
Using
	$$
	\left|\frac{f^{(k)}(z)}{f^{(j)}(z)}\right|
	=\left|\frac{f^{(k)}(z)}{f^{(k-1)}(z)}\right|\cdots\left|\frac{f^{(j+1)}(z)}{f^{(j)}(z)}\right|
	$$
together with \eqref{derivative} and the first part of the proof, the assertion follows. 
\end{proof}

\begin{remark}
Taking $R= r+ 1/ T(r,f)$ in Lemma~\ref{G-estimate} and using Borel's Lemma \cite[Lemma~2.4,]{H}, we obtain that there exists a set $E_4\subset [0,\infty)$ of finite linear measure, such that
	\begin{equation}\label{remark1}
	\log^+\left|\frac{f^{(k)}(z)}{f^{(j)}(z)}\right| \lesssim \log T(r,f) + \log r, \quad r\notin E_4.
	\end{equation}
\end{remark}

\begin{proof}[Proof of Theorem~\ref{asymptotic-cor}]
We prove the theorem in three steps.

(i) Let $\{f_{0,1},\ldots,f_{0,n} \}$ be a given solution base of \eqref{lden}. We prove that there exist at least $n-p$ solutions $f$ in $\{f_{0,1},\ldots,f_{0,n}\}$ and a set $E\subset [0,\infty)$ of finite linear measure such that
	\begin{equation}\label{in2}
	\log M(r,A_p)\lesssim \log T(r,f),\quad r\not \in E.
	\end{equation} 
It suffices to prove that there are at most $p$ solutions $f$ 
in $\{f_{0,1},\ldots,f_{0,n}\}$ and a set $F\subset [0,\infty)$ of infinite linear measure such that 
	\begin{equation}\label{3.1}
	\log T(r,f)=o(\log M(r,A_p)),\quad r\to\infty,\ r\in F.
	\end{equation}
 We assume on the contrary to this claim that there are $p+1$ solutions $f$ in $\{f_{0,1},\ldots,f_{0,n}\}$, say $f_{0,1},\ldots,f_{0,{p+1}}$, each satisfying \eqref{3.1}, and aim for a contradiction.

Note that $A_p$ is transcendental, because if this is not the situation, that is, if $A_p$ is a polynomial, then by \eqref{2LM2} we deduce that $A_{p+1},\ldots,A_{n-1}$ are also polynomials. If $A_{p-1}$ is transcendental, then
	$$
	\limsup_{r\to \infty} \sum_{j=p}^{n-1}\frac{\log^+ M(r,A_{j})}{\log^+ M(r,A_{p-1})}=0,
	$$	
which contradicts the assumption that $p$ is the smallest index for which \eqref{2LM2} holds. Thus $A_{p-1}$ is also polynomial. Similarly it follows that  $A_0,\ldots,A_{p-2}$ are polynomials. But this contradicts the assumption that at least one of the coefficients is transcendental. 

From Lemma~\ref{coeff_Ap}, we have
	\begin{align}
	\log^+ |A_p(z)| \le  \sum_{j=p+1}^{n-1} \log M(r,A_j) + \sum_{k=p+1}^n \log^+|C_k(z)|  + O(1). \label{coeff-estimate}
	\end{align}
It follows from \eqref{CC} and \eqref{remark1} together with Lemma~\ref{5.3} that
	\begin{equation}\label{log-estimate}
	\begin{split}
	\log^+ |C_k(z)| &= O\left( \sum_{l_0+\cdots+l_p=k-p}\; \sum_{\nu=0}^p \log^+ \left|\frac{f_{\nu,1}^{(l_\nu)}}{f_{\nu,1}}\right| + 1\right) \\
	 & = O\left( \sum_{\nu=0}^p \log T(r,f_{\nu,1})  + \log r\right) \\
	 & = O\left(\sum_{\nu=0}^p \; \sum_{l=1}^{\nu+1} \log T(r,f_{0,l}) + \log r\right) \\
	 &= O\left( \sum_{l=1}^{p+1} \log T(r,f_{0,l}) + \log r \right), \quad r=|z| \notin (E_1\cup E_4). 
	\end{split}
	\end{equation}
Therefore, we get from \eqref{coeff-estimate} and \eqref{log-estimate} that  
	$$
	\log M(r,A_p) \le  \sum_{j=p+1}^{n-1} \log M(r,A_j) + O \left(\sum_{l=1}^{p+1} \log T(r,f_{0,l}) + \log r \right), \quad r\notin ( E_1\cup E_4).
	$$
Since $F$ has infinite linear measure, it follows that $F\setminus (E_1\cup E_4)$ has also infinite linear measure. Then, using \eqref{2LM2}, \eqref{3.1} and the fact that $A_p$ is transcendental, we obtain
	\begin{align*}
	1 & \le \limsup_{\substack{r\to\infty \\ r\in F\setminus (E_1\cup E_4)}}\frac{\sum_{j=p+1}^{n-1} \log M(r,A_j)}{\log M(r,A_p)} 
	+ \limsup_{\substack{r\to\infty \\ r\in F\setminus (E_1\cup E_4)}} O \left( \frac{\sum_{l=1}^{p+1} \log T(r,f_{0,l})}{\log M(r,A_p)} +\frac{\log r}{\log M(r,A_p)} \right) \\
	& \le \limsup_{\substack{r\to\infty}}\frac{\sum_{j=p+1}^{n-1} \log M(r,A_j)}{\log M(r,A_p)} 
	+ \limsup_{\substack{r\to\infty \\ r\in F}} O \left( \frac{\sum_{l=1}^{p+1} \log T(r,f_{0,l})}{\log M(r,A_p)} +\frac{\log r}{\log M(r,A_p)} \right) \\
	& <1,
	\end{align*}
which is absurd. Thus, the asymptotic inequality \eqref{in2} is now proved.

(ii) We prove that any non-trivial solution $f$ of \eqref{lden} satisfies
	\begin{equation}\label{in1}
	\log T(r,f)  \lesssim \log M(r,A_p).
	\end{equation}
From \cite[Corollary~5.3]{HKR}, we infer
	\begin{equation}\label{T-max}
	T(r,f)=m(r,f)\lesssim  r\, \sum_{j=0}^{n-1}M(r,A_j)^\frac{1}{n-j}+1, \quad r\ge 0.
	\end{equation}
From \eqref{2LM2} and  \eqref{T-max}, we obtain
	\begin{equation}\label{logT}
	\begin{split}
	\log^+ T(r,f) &\lesssim \log^+ r+\sum_{j=0}^{n-1}\log^+ M(r,A_j)\\
	&\lesssim \sum_{j=0}^{p-1}\log^+ M(r,A_j)+\log^+ M(r,A_p),
	\end{split}
	\end{equation}
where the sum on the right is empty if $p=0$. Hence we suppose that $p\geq 1$.

We proceed to prove that
	\begin{equation}\label{DM}
	\log^+ M(r,A_j)\lesssim \log^+ M(r,A_{p}),\quad 0\leq j\leq p-1.
	\end{equation}
Suppose on the contrary to this claim that there exists an $s\in\{0,\dots,p-1\}$ such that
	\begin{equation}\label{D-2LM}
	{\limsup_{r\to \infty}}\frac{\log^+ M(r,A_p)}{\log^+ M(r,A_s)}=0.
	\end{equation}
Choose $s$ to be the largest index in $\{0,\dots,p-1\}$ for which \eqref{D-2LM} occurs.
If $s=p-1$, we arrive at a contradiction with the definition of the index $p$.
Thus $s\in \{0,\ldots,p-2\}$, where $p\geq 2$. 
Then \eqref{DM} holds for $j=s+1,\ldots,p-1$. Hence, from \eqref{2LM2} and \eqref{D-2LM}, we obtain
	\begin{align}\label{3LM}
	&\limsup_{r\to \infty}\frac{\sum_{j=s+1}^{n}\log^+ M(r,A_{j})}{\log^+ M(r,A_{s})}\nonumber\\
	&=\limsup_{r\to \infty}\frac{\sum_{j=s+1}^{p-1}\log^+ M(r,A_{j})+\sum_{j=p+1}^{n}\log^+ M(r,A_{j})
	+\log^+ M(r,A_{p})}{\log^+ M(r,A_{s})} = 0,\nonumber
	\end{align}
which contradicts our assumption that $p$ is the smallest index for which \eqref{2LM2} occurs. 
This proves \eqref{DM}.  Thus \eqref{in1} follows from \eqref{logT} and \eqref{DM}.

(iii) It remains to prove that $0$ is the only finite deficient value for the rapid solutions. According to Wittich's theorem, it suffices to prove that rapid solutions are also admissible solutions of \eqref{lden}. From \eqref{2LM2} and \eqref{DM} we get $\log M(r,A_j) \lesssim \log M(r,A_{p})$, for every $j=0, \ldots, n-1$. Thus, every rapid solution $f$ of \eqref{lden} satisfies, for all $j=0, \ldots, n-1$,
	$$
	\frac{T(r,A_j)}{T(r,f)} \le \frac{\log M(r,A_j)}{T(r,f)} \lesssim \frac{\log M(r,A_{p})}{T(r,f)} \asymp \frac{\log T(r,f)}{T(r,f)} \to 0, \quad r\to \infty, \  r\notin E,
	$$
i.e., every rapid solution is admissible solution.
\end{proof}


\begin{proof}[Proof of Theorem~\ref{sdsd}]
Similarly to the proof of Theorem~\ref{asymptotic-cor}, we assume the contrary to the assertion that there exist $p+1$ linearly independent solutions $f_{0,1},\ldots,f_{0,p+1}$  and a set $F\subset[0,\infty)$ of infinite linear measure such that
	\begin{equation}\label{p1}
	\log T(r,f_{0,l})=o(\log M(r,A_{p})),\quad r\to\infty,\  r\in F,\  l=1,\ldots, p+1.
	\end{equation}
From Lemma~\ref{coeff_Ap} and from Young's inequality for the product \cite[p.~49]{M}, we obtain
	\begin{equation}\label{modu-A-p}
	|A_{p}(z)| \le  \sum_{j=p+1}^{n-1} \frac{1}{\eta_j}{|A_{j}(z)|^{\eta_j}} + |C_n(z)| + \sum_{j=p+1}^{n-1} |C_{j}(z)|^{\eta_j^*} , 
	\end{equation}
where the constants $\eta_j>1$ are given in the statement of the theorem, and the constants $\eta_j^*>1$ are their conjugate indices satisfying $1/\eta_j + 1/\eta_j^* =1$ for every $j=p+1, \ldots, n-1$. 
From \eqref{p-u-1} we can find a $\delta>0$ such that for some $r_0>0$ we have
	\begin{equation}\label{modu-right-1}
	\sum_{j=p+1}^{n-1} \frac{1}{\eta_j}{|A_{j}(z)|^{\eta_j}} <(1-\delta) M(r,A_p), \quad z\in \Gamma,\; |z|=r > r_0.
	\end{equation}
Hence, it follows from \eqref{log-estimate}, \eqref{modu-A-p} and \eqref{modu-right-1} that
	\begin{equation*}
 \log M(r,A_{p}) \lesssim \sum_{l=1}^{p+1} \log T(r,f_{0,l})+\log r, \quad r\in (r_0, \infty) \setminus (E_1\cup E_4).
 \end{equation*}
Dividing both sides of the last asymptotic inequality by $\log M(r,A_p)$ and by letting $r\to\infty$ in $F\setminus (E_1\cup E_4)$ and using \eqref{p1} and the fact that $A_p$ is transcendental, we get a contradiction. 
Thus the proof is complete. 
\end{proof}

\begin{proof}[Proof of Theorem~\ref{thhh1}]
The proof is quite similar to the proof of Theorem~\ref{asymptotic-cor}. Hence, we only state the differences and omit the rest of the details.

For the lower bound of $\log T(r,f)$ in \eqref{3'}, we apply the proximity function on \eqref{4} in Lemma~\ref{coeff_Ap}, and use the standard logarithmic derivative estimate.

We deduce the upper bound of $\log T(r,f)$ in \eqref{3'} in the following way:
Similarly to \eqref{DM}, we deduce 
	$$
	T(r,A_j) \lesssim T(r,A_{p}),\quad 0\leq j\leq p-1.
	$$
Therefore, combining this with \eqref{T-max} and \eqref{2} and the fact that $A_p$ is transcendental, we obtain for any $r<R<\infty$,
	\begin{align*}
	\log^+ T(r,f) &\lesssim \log^+ r+\sum_{j=0}^{n-1}\log^+ M(r,A_j)\\
	& \le \log^+R + \frac{R+r}{R-r} \sum_{j=0}^{n-1} T(R,A_j)\\
	&\lesssim  \frac{R+r}{R-r} T(R,A_p). 
	\end{align*}
Finally, every rapid solution $f$ of \eqref{lden} satisfies
	$$
	\frac{T(r,A_j)}{T(r,f)} \lesssim \frac{T(r,A_p)}{T(r,f)} \lesssim \frac{\log T(r,f)}{T(r,f)} \to 0, \quad r\to\infty,\ r\notin E,
	$$
for all $j=0,\ldots, n-1$.
\end{proof}

\section{Proofs of the results in the unit disc}\label{proof-disc}

The proofs of Theorems \ref{th.1-ud-1}--\ref{pointwise th} follow their plane analogues. In fact, we use Lemma~\ref{5.3-unit} instead of Lemma~\ref{5.3}. Furthermore, we use the unit disc counterpart of the lemma on the logarithmic derivatives to prove Theorem~\ref{th.1-ud-1}. The following lemma is the unit disc analogue of Lemma~\ref{G-estimate} and is needed to prove Theorems~\ref{asymptotic-cor-u1} and \ref{pointwise th}.

\begin{lemma}\label{G-estimate-disc}
Let $0<R<1$, $\alpha>1$, and let $f$ be a meromorphic function in $\D$. Suppose that $k,j$ are integers with $k>j\geq 0$, and $f^{(j)}\not\equiv 0$. Then there exists a~set $E\subset[0,1)$ with $\int_E \frac{dr}{1-r}<\infty$ such that for all $z$ satisfying $|z|=  r \in (0,R) \setminus E$, we have 
	\begin{equation}\label{G-estimate-2}
	\left|\frac{f^{(k)}(z)}{f^{(j)}(z)}\right|
	\lesssim \frac{1}{(R-r)^{k-j}}\left\{  \frac{R}{R-r}\left(1+\log ^{+} \frac{1}{R-r}+T(R, f)\right)\right\}^{(1+\alpha)(k-j)}. 
	\end{equation}
Moreover, if $k=1$ and $j=0$, then the logarithmic term in \eqref{G-estimate-2} can be omitted.
\end{lemma}

\begin{proof}
Following the proof of Lemma~\ref{G-estimate}, let $U$ be the collection of discs $D(a_m, R_m)$, where $R_m = (1-|a_m|) / m^\alpha$ and $\{a_m\}$ is the sequence of zeros and poles of $f^{(j)}$ in $\D$ listed according to multiplicity and ordered by increasing modulus. Clearly,
	$$
	\sum_{m=1}^\infty \frac{R_m}{1-|a_m|} <\infty.
	$$
Then the projection $E$ of $U$ on $[0,1)$ satisfies $\int_E \frac{dr}{1-r}< \infty$, see \cite[pp.~749-750]{CGH}. Let $L$ denote the number of the points $a_m$ at the origin. If $z \notin U$, we have
	$$
	\sum_{|a_m|<\varrho}\frac{1}{|z-a_m|}\leq L+\sum_{0<|a_m|<\varrho}\frac{m^\alpha}{1-|a_m|}
	\leq L+\frac{1}{R-\varrho}\sum_{0<|a_m|<\varrho}n(\varrho)^\alpha\leq L+\frac{n(\varrho)^{1+\alpha}}{R-\varrho},
	$$
for all $r< \varrho <R$. Using this estimate with the other estimates in the proof of Lemma~\ref{G-estimate}, the assertion follows.
\end{proof}

We will also use the following minor modification of Borel's lemma \cite[Lemma~2.4]{H}.

\begin{lemma} \label{Borel}
Let $T: [r_0, 1) \mapsto [1,\infty)$ be continuous and non-decreasing function. For any $\sigma>0$, there exists a set $E(\sigma) \subset [0,1)$ with $\int_{E(\sigma)} \frac{dt}{1-t}< \infty$ such that
	\begin{equation*}
	T\left( r + \frac{1-r}{e T(r)^{\sigma}} \right) < 2 T(r),\quad r\notin E(\sigma).
	\end{equation*}
\end{lemma}

\begin{remark}
Taking $R=r+(1-r) / (e T(r,f))$ in Lemma~\ref{G-estimate-disc}, and using Lemma~\ref{Borel}, we get
	$$
	\log \left|\frac{f^{(k)}(z)}{f^{(j)}(z)}\right| \lesssim \log T(r,f) + \log \frac{1}{1-r}, \quad r\notin E,
	$$
where $\int_E \frac{dt}{1-t} < \infty$.
\end{remark}

To prove Theorem~\ref{sdsd2-disc}, we will use an estimation for the logarithmic derivatives from \cite{CGHR}.

\begin{lemma}[\cite{CGHR}]\label{CGHR-estimate}
Let $0<R<\infty$ and let $f$ be meromorphic in a domain containing $\overline{D(0,R)}$. Suppose that $j,k$ are integers with $k>j\ge 0$, and $f^{(j)}\not\equiv 0$. Then
	\begin{align*}
	\int_{r^{\prime}<|z|<r} & \left|\frac{f^{(k)}(z)}{f^{(j)}(z)}\right|^{\frac{1}{k-j}} d m(z)\\
	&\lesssim R \log \frac{ e\left(R-r^{\prime}\right)}{R-r}\left(1+\log ^{+} \frac{1}{R-r}+T(R, f)\right), \quad 0 \leq r^{\prime}<r<R.
	\end{align*}
\end{lemma}

\begin{remark}
In the case that $f$ is meromorphic in $\D$, we take $R=r+(1-r) / (e T(r,f))$ in Lemma~\ref{CGHR-estimate} and use Lemma~\ref{Borel} to obtain
	\begin{equation}\label{remark2}
	\log^+\int_{D(0,r)}  \left|\frac{f^{(k)}(z)}{f^{(j)}(z)}\right|^{\frac{1}{k-j}} d m(z) \lesssim \log T(r,f) + \log \frac{1}{1-r}, \quad r\notin E_3,
	\end{equation}
where $E_5$ is a set with $\int_{E_5} \frac{dt}{1-t}<\infty$.
\end{remark}


\begin{proof}[Proof of Theorem~\ref{sdsd2-disc}]

(i) Let $\{f_{0,1},\ldots,f_{0,n} \}$ be a given solution base of \eqref{lden}. We prove that there exist at least $n-p$ solutions $f$ in $\{f_{0,1},\ldots,f_{0,n}\}$ and a set $E\subset [0,1)$, with $\int_E \frac{dt}{1-t} < \infty$, such that
	\begin{equation}\label{pr1}
	\log^+\int_{D(0,r)} \left| A_p(z) \right|^{\frac{1}{n-p}} dm(z)  \lesssim \log T(r,f), \quad r\notin E.
	\end{equation} 
It suffices to prove that there are at most $p$ solutions $f$ 
in $\{f_{0,1},\ldots,f_{0,n}\}$ and a set $F\subset [0,1)$, with$\int_F \frac{dt}{1-t} = \infty$, such that 
	\begin{equation}\label{p11}
	\log T(r,f) = o\left( \log^+\int_{D(0,r)} \left| A_p(z) \right|^{\frac{1}{n-p}} dm(z) \right),\quad r\to1^-, \  r\in F,
	\end{equation}
 We assume on the contrary to this claim that there are $p+1$ solutions $f$ in $\{f_{0,1},\ldots,f_{0,n}\}$, say $f_{0,1},\ldots,f_{0,{p+1}}$, each satisfying \eqref{p11}, and aim for a contradiction.

Notice that $A_p \notin \A^{-\infty}$. In fact, if $A_p \in \A^{-\infty}$, then clearly
	$$
	\log\int_{D(0,r)} |A_p(z)|^{\frac{1}{n-p}} dm(z) \lesssim \log \frac{1}{1-r}.
	$$
Hence, from \eqref{con1}, the discussion following Theorem~\ref{sdsd2-disc} and the definition of the index $p$, we deduce that all coefficients $A_0, \ldots,A_{n-1}$ belong to $\A^{-\infty}$ and this contradicts the assumption that at least one  coefficient is not in $\A^{-\infty}$.

It follows from Lemma~\ref{coeff_Ap}, that
	$$
	|A_p(z)|^{\frac{1}{n-p}} \le \sum_{j=p+1}^{n-1} |A_j(z)|^{\frac{1}{n-p}} |C_j(z)|^{\frac{1}{n-p}} + |C_n(z)|^{\frac{1}{n-p}}.
	$$
Using Young's inequality for the product \cite[p.~49]{M} with conjugate indices $\frac{n-p}{n-j}$ and $\frac{n-p}{j-p}$, we get
	\begin{equation}\label{f1}
	|A_p(z)|^{\frac{1}{n-p}} \le  \sum_{j=p+1}^{n-1} \frac{n-j}{n-p} |A_j(z)|^{\frac{1}{n-j}} + \sum_{k=p+1}^n |C_k(z)|^{\frac{1}{k-p}}.
	\end{equation}
From \eqref{con1}, we deduce that there is a $\delta >0$ sufficiently small and an $r_0>0$ such that 
	\begin{equation}\label{f11}
	\sum_{j=p+1}^{n-1} \frac{n-j}{n-p}\int_{D(0,r)} |A_j(z)|^{\frac{1}{n-j}} dm(z) < (1-\delta) \int_{D(0,r)} |A_p(z)|^{\frac{1}{n-p}} dm(z), \quad r>r_0.
	\end{equation}
By combining \eqref{f1} and \eqref{f11}, we obtain
	\begin{equation}\label{f12}
	\log^+ \int_{D(0,r)} |A_p(z)|^{\frac{1}{n-p}} dm(z) \lesssim \sum_{k=p+1}^n  \log^+\int_{D(0,r)} |C_k(z)|^{\frac{1}{k-p}} dm(z) + 1, \quad r>r_0.
	\end{equation}
For $k=p+1, \dots, n$, we have by using the weighted AM--GM inequality \cite[p.~22]{P}, \eqref{remark2} and Lemma~\ref{5.3-unit},
	\begin{align*}
	\log^+\int_{D(0,r)}|C_k (z)|^{\frac{1}{k-p}} dm(z) &\lesssim  \sum_{l_{0}+l_{1}+\cdots+l_{p}=k-p} \log^+\int_{D(0,r)}\left(\left|\frac{f_{0,1}^{\left(l_{0}\right)}}{f_{0,1}}\right|\left|\frac{f_{1,1}^{\left(l_{1}\right)}}{f_{1,1}}\right| \cdots\left|\frac{f_{p, 1}^{\left(l_{p}\right)}}{f_{p, 1}}\right| \right)^{\frac{1}{k-p}} dm(z) + 1  \nonumber\\
	& \lesssim \sum_{l_{0}+l_{1}+\cdots+l_{p}=k-p}\; \sum_{\nu=0}^p   \log^+\int_{D(0,r)} \left|\frac{f_{\nu,1}^{\left(l_\nu\right)}}{f_{\nu,1}}\right|^{\frac{1}{l_\nu}} dm(z) +1 \nonumber \\
	& \lesssim \sum_{\nu=0}^p \log T(r,f_{\nu,1})  + \log \frac{1}{1-r} \nonumber \\
	& \lesssim   \sum_{l=1}^{p+1} T(r,f_{0,l}) + \log \frac{1}{1-r} , \quad r=|z|\in (r_0, 1) \setminus (E_2\cup E_5). \label{f2}
	\end{align*}
Hence, from \eqref{f12}, we get
	$$
	\log^+ \int_{D(0,r)} |A_p(z)|^{\frac{1}{n-p}} dm(z) \lesssim  \sum_{l=1}^{p+1} T(r,f_{0,l}) + \log \frac{1}{1-r}, \quad r\in (r_0, 1) \setminus (E_2\cup E_5).
	$$
By letting $r\to 1^-$ in $F\setminus(E_2\cup E_5)$, we get a contradiction by \eqref{p11} and \eqref{int-tran}. Thus, \eqref{pr1} is proved.

(ii) We prove that any non-trivial solution $f$ of \eqref{lden} satisfies
	\begin{equation}\label{pr2}
	\log T(r,f) \lesssim \log \int_{D(0,r)} |A_p(z)|^{\frac{1}{n-p}} dm(z).
	\end{equation}
From \cite[Corollary~5.3]{HKR}, see also \cite[Lemma~F]{HKR3}, we obtain
	\begin{align*}
	T(r,f) \lesssim \sum_{j=0}^{n-1} \int_{D(0,r)} |A_j(z)|^{\frac{1}{n-j}} dm(z)+1, 
	\end{align*}
and using \eqref{con1} yields
	\begin{equation}\label{prr1}
	T(r,f) \lesssim \sum_{j=0}^{p-1} \int_{D(0,r)} |A_j(z)|^{\frac{1}{n-j}} dm(z) + \int_{D(0,r)} |A_p(z)|^{\frac{1}{n-p}} dm(z).
	\end{equation}
With the same reasoning used to prove \eqref{DM}, we can prove 
	\begin{equation}\label{prr2}
	\int_{D(0,r)} |A_j(z)|^{\frac{1}{n-j}} dm(z) \lesssim \int_{D(0,r)} |A_p(z)|^{\frac{1}{n-p}} dm(z), \quad j=0,\ldots, p-1.
	\end{equation}
Thus, \eqref{pr2} follows from \eqref{prr1} and \eqref{prr2}.

(iii) The fact that the solutions satisfying \eqref{pr1} are rapid in the sense of (II) follows immediately from \eqref{int-tran} with $A_p$ in place of $g$. Thus, it remains to prove that $0$ is the only possible finite deficient value of the solutions $f$ satisfying \eqref{pr1}.
Let $f$ be a non-trivial solution of \eqref{lden} satisfying \eqref{pr1}. For any $j=0, \ldots, n-1$, $0<r<R<1$ and $R>1/\sqrt{\pi}$, we obtain by Jensen's inequality 
	\begin{equation}\label{transc}
	\begin{split}
	\log^+ \int_{D(0,R)} |A_j(z)|^{\frac{1}{n-j}} dm(z) &\ge \log^+ \int_{D(0,R)} |A_j(z)|^{\frac{1}{n-j}} \frac{dm(z)}{\pi 
R^2}  \\
	&\ge  \frac{1}{(n-j)\pi R^2}\int_{D(0,R)} \log^+|A_j(z)|dm(z) \\
	&\ge  \frac{1}{(n-j)\pi R^2} \int_{r}^R T(t,A_j)  t dt \\
	&\ge \frac{r}{(n-j)\pi R^2} \; (R-r)T(r,A_j). 
	\end{split}
	\end{equation}
From \eqref{con1} and \eqref{prr2}, we have 
	\begin{equation*}
	\log^+\int_{D(0,R)} |A_j(z)|^{\frac{1}{n-j}} dm(z) \lesssim  \log^+\int_{D(0,R)} |A_p(z)|^{\frac{1}{n-p}} dm(z)+1, \quad j=0,\ldots, n-1.
	\end{equation*}
Therefore, combining this with \eqref{pr1} and \eqref{transc}, it follows
	\begin{equation}\label{coefs}
	T(r,A_j) \lesssim \frac{\log T(R,f)+1}{R-r}  ,\quad  R\notin E, \ j=0, \ldots, n-1.
	\end{equation}
Let $R=r+ \frac{1-r}{e \sqrt{T(r,f)}}$, and let $\tilde{E}= \{r \in [0,1) : R \in E\}$. We will prove that $\int_{\tilde{E}} \frac{dr}{1-r} < \infty$. Suppose that $\int_{\tilde{E}} \frac{dr}{1-r} =\infty$, and aim for a contradiction. We have
	\begin{equation}\label{CV}
	dR = \varphi(r) dr,
	\end{equation}
where
	$$
	\varphi(r)= 1 - \frac{1}{e\sqrt{T(r,f)}} - \frac{(1-r) \frac{dT(r,f)}{dr}}{2eT(r,f)^{3/2}}.
	$$
Using the first main theorem \cite[Theorem~2.1.10]{Laine} in \eqref{n}, we obtain for any large enough $r<1$ and any $r< R^* <1$, that
	$$
	n(r,e^{i\theta}, f) \le 2 \frac{T(R^*,f)}{R^*-r}
	$$
uniformly for any $\theta\in[0,2\pi]$.
Therefore, by choosing $R^* = r + \frac{1-r}{e T(r,f)^{1/4}}$, it follows from Lemma~\ref{Borel} and from Cartan's identity \cite[p.~9]{H}, that
	$$
	r \frac{dT(r,f)}{dr} = \frac{1}{2\pi} \int_0^{2\pi} n(r,e^{i\theta}, f) d\theta \le 4e \frac{T(r,f)^{5/4}}{1-r}, \quad r\notin E(\tfrac{1}{4}),
	$$
where $\int_{E(\frac{1}{4})} \frac{dt}{1-t}< \infty$. Thus, we obtain
	$$
	1 > \varphi(r) \ge  1 - \frac{1}{e\sqrt{T(r,f)}} - \frac{2}{r\sqrt[4]{T(r,f)}}, \quad r\notin E(\tfrac14).
	$$
Hence, there exists an $r_0\in(0,1)$ such that for all $r\in (r_0,1) \setminus E(\tfrac14)$, we have $1/2 <\varphi(r)< 1$.
Thus,
	$$
	\infty > \int_E \frac{dR}{1-R} \ge  \int_{\tilde{E}} \frac{\varphi(r)dr}{1-r} \ge \frac12 \int_{\tilde{E} \setminus E(\frac14)} \frac{dr}{1-r} = \infty,
	$$
which is a contradiction. Hence, $\int_{\tilde{E}} \frac{dr}{1-r} < \infty$.

From \eqref{coefs} and from Lemma~\ref{Borel} by choosing $R=r+ \frac{1-r}{e \sqrt{T(r,f)}}$,  we obtain for $r\notin \big(\tilde{E} \cup E(\tfrac12) \big)$,
	$$
	\frac{T(r,A_j)}{T(r,f)} \lesssim \frac{\log T(R,f)+1}{(R-r)T(r,f)} \lesssim \frac{\log T(r,f)+1}{(1-r) \sqrt{T(r,f)}}  \lesssim \frac{1}{(1-r) \sqrt[4]{T(r,f)}}.
	$$
Since $f$ is of infinite order, it follows that there exists a set $F\subset[0,1)$ with $\int_F \frac{dt}{1-t}=\infty$ such that
	$$
	T(r,f)> \frac{1}{(1-r)^8}, \quad r\in F.
	$$
Therefore, 
	$$
	\frac{T(r,A_j)}{T(r,f)} \lesssim 1-r, \quad r\in F\setminus \big(\tilde{E} \cup E(\tfrac12)\big),
	$$
which means, $T(r,A_j) = o (T(r,f))$, $r\in F\setminus \big(\tilde{E} \cup E(\tfrac12)\big)$, for any $j=0, \ldots, n-1$. Clearly, the set $\tilde{F}= F\setminus \big(\tilde{E} \cup E(\tfrac12)\big)$ satisfies $\int_{\tilde{F}} \frac{dt}{1-t} = \infty$. Thus, following the proof of Wittich's theorem \cite[Theorem~4.3]{Laine}, we deduce that for any $a\in\C\setminus \{0\}$
	$$
	m\left(r,\frac{1}{f-a} \right) = o \left(T(r,f) \right),  \quad r\to 1^-,\  r \in \tilde{F}.
	$$
Therefore,
	\begin{equation*}
	\delta(a,f)=\liminf_{r\to1^-} \frac{m\left(r,\frac{1}{f-a} \right)}{T(r,f)} \le \liminf_{\substack{r\to1^- \\ r\in \tilde{F}}} \frac{m\left(r,\frac{1}{f-a} \right)}{T(r,f)}=0. 
	\end{equation*}
Hence, $0$ is the only possible finite deficient value for $f$.
\end{proof}


The following lemma is needed to prove Theorem~\ref{s-int}.

\begin{lemma}[{\cite{CGHR}}] \label{C-estimate}
Let $f$ be meromorphic in $\D$, and let $j,k$ be integers with $k>j\geq 0$ such that $f^{(j)}\not\equiv 0$. Let $s:[0,1) \to [0,1)$ be an~increasing continuous function such that $s(r) \in (r,1)$ and $s(r)-r$ is decreasing. If $\delta\in (0,1)$, then there exists a~measurable set $E\subset [0,1)$ with
		$
        \overline{d}(E)\leq \delta
        $
such that
	\begin{equation}\label{1}
  	\int_0^{2\pi} \bigg| \frac{f^{(k)}(re^{i\theta})}{f^{(j)}(re^{i\theta})} \bigg|^{\frac{1}{k-j}} \, d\theta
  \lesssim \frac{T(s(r),f) -\log(s(r)-r)}{s(r)-r}, \quad r\notin E.
	\end{equation}
Moreover, if $k=1$ and $j=0$, then the logarithmic term in \eqref{1} can be omitted.
\end{lemma}

\begin{remark}
Since $s(r)$ in Lemma~\ref{C-estimate} is arbitrary, we can choose it as $s(r)=r+ (1-r)/(eT(r,f))$. Then, using Lemma~\ref{Borel}, we obtain
	\begin{equation*}
	\log \int_0^{2\pi} \bigg| \frac{f^{(k)}(re^{i\theta})}{f^{(j)}(re^{i\theta})} \bigg|^{\frac{1}{k-j}} \, d\theta \lesssim \log T(r,f) + \log \frac{1}{1-r}, \quad r\notin E,
	\end{equation*}
where $\overline{d}(E)<1$.
\end{remark}

\begin{proof}[Proof of Theorem~\ref{s-int}]
 Let $\{f_{0,1},\ldots,f_{0,n} \}$ be a given solution base of \eqref{lden}. We prove that there exist at least $n-p$ solutions $f$ in $\{f_{0,1},\ldots,f_{0,n}\}$ and a set $E\subset [0,1)$, with $\overline{d}(E)<1$, such that
	\begin{equation}\label{pr1-1}
	\log^+\int_{0}^{2\pi} \left| A_p(re^{i\theta}) \right|^{\frac{1}{n-p}} d\theta  \lesssim \log T(r,f), \quad r\notin E.
	\end{equation} 
It suffices to prove that there are at most $p$ solutions $f$ 
in $\{f_{0,1},\ldots,f_{0,n}\}$ and a set $F\subset [0,1)$, with $\overline{d}(F) = 1$, such that 
	\begin{equation*}
	\log T(r,f) = o\left( \log^+\int_0^{2\pi} \left| A_p(re^{i\theta}) \right|^{\frac{1}{n-p}} d\theta \right),\quad r\to1^-, \ r\in F.
	\end{equation*}
 We assume on the contrary to this claim that there are $p+1$ solutions $f$ in $\{f_{0,1},\ldots,f_{0,n}\}$, say $f_{0,1},\ldots,f_{0,{p+1}}$, each satisfying \eqref{p11}, and aim for a contradiction. The rest of the proof is similar to the proof of Theorem~\ref{sdsd2-disc}, and hence we omit the details here.

The fact that the solutions satisfying \eqref{pr1-1} are rapid in the sense of (I) follows from \eqref{r.solI}. Thus, it remains to prove that $0$ is the only finite deficient value for the solutions satisfying \eqref{pr1-1}. From \eqref{limsup-int} and the definition of the index $p$, we get
	$$
	\log^+\int_{0}^{2\pi} \left| A_j(re^{i\theta}) \right|^{\frac{1}{n-j}} d\theta \lesssim \log^+\int_{0}^{2\pi} \left| A_p(re^{i\theta}) \right|^{\frac{1}{n-p}} d\theta,\quad j=0, \ldots, n-1.
	$$
Therefore, for any solution $f$ satisfying \eqref{pr1-1} and for every $j=0, \ldots, n-1$, we have by Jensen's inequality,
	\begin{align*}
	\frac{T(r,A_j)}{T(r,f)} &\lesssim \frac{\log^+ \displaystyle\int_{0}^{2\pi} \left| A_j(re^{i\theta}) \right|^{\frac{1}{n-j}} d\theta}{T(r,f)}\\
	&\lesssim 	\frac{\log^+ \displaystyle\int_{0}^{2\pi} \left| A_p(re^{i\theta}) \right|^{\frac{1}{n-p}} d\theta}{T(r,f)} \lesssim \frac{\log T(r,f)}{T(r,f)} \to 0, \quad r\to 1^-,\  r\notin E.
	\end{align*}
Following the proof of Wittich's theorem \cite[Theorem~4.3]{Laine}, we deduce that for any $a\in\C\setminus \{0\}$
	$$
	m\left(r,\frac{1}{f-a} \right) = o \left(T(r,f) \right),  \quad r\to 1^-,\  r\notin E,
	$$
where $\overline{d}(E)<1$. Thus,
	\begin{equation*}
	\delta(a,f)=\liminf_{r\to1^-} \frac{m\left(r,\frac{1}{f-a} \right)}{T(r,f)} \le \liminf_{\substack{r\to1^- \\ r\notin E}} \frac{m\left(r,\frac{1}{f-a} \right)}{T(r,f)}=0. 
	\end{equation*}
Hence, $0$ is the only possible finite deficient value for $f$.
\end{proof}


\section{Results on linear difference equations}\label{difference-sec}

Consider the difference equation
	\begin{equation}\label{d1}
	\Delta^n f(z)+A_{n-1}(z)\Delta^{n-1} f(z)+\cdots+A_{1}(z)\Delta f(z)+A_{0}(z)f(z)=0,
	\end{equation}
where $A_0(\not\equiv 0),\ldots,A_{n-1}$ are entire functions, and $\Delta $ is a difference operator defined by $\Delta f(z)=f(z+1)-f(z)$ and $\Delta^n f(z) = \Delta (\Delta^{n-1} f(z))$. Equation \eqref{d1} can be written in the form
	\begin{equation}\label{CF-diff}
	f(z+n)+B_{n-1}(z)f(z+n-1)+\cdots+B_0(z)f(z)=0
	\end{equation}	
and vice versa, see \cite[Section~3.2]{RO}.  Concerning the growth of meromorphic solutions of \eqref{CF-diff}, where the coefficients $B_{0},\ldots,B_{n-1}$ are entire, Korhonen and Ronkainen proved \cite[Theorem~4]{RO}, which reads as follows: \emph{Suppose that there exists an integer $p$ such that 
	\begin{equation}\label{assumption}
	T(r,B_j) = o (T(r,B_p))
	\end{equation}
for all $j\neq p$, where $r\to\infty$ outside an exceptional set of finite logarithmic measure. If $f$ is a meromorphic solution of \eqref{CF-diff} with hyper-order $\rho_2(f)=\rho_2<1$, then for any $\veps>0$,
	   \begin{equation}\label{conclusion-di}
	   T(r,f) \ge  r^{1-\rho_2-\veps} T(r,B_p)
	   \end{equation}
outside a set of finite logarithmic measure}. This result is a generalization of \cite[Theorem~9.2]{CF}. From the proof of \cite[Theorem~4]{RO}, if  we replace the assumption  \eqref{assumption} by
	 \begin{equation*}\label{condition-dom}
	 \limsup_{r\to \infty} \sum_{j\neq p}\frac{T(r,B_j)}{T(r,B_p)} <1,
	 \end{equation*}
then the same conclusion \eqref{conclusion-di} still holds. 
However, the reasoning in the proof of \cite[Theorem~4]{RO} does not seem to apply to \eqref{d1}. The reason for this is that the estimate for $m(r,\Delta^j f / \Delta^k f)$, $j < k$, is different from the corresponding estimate for $m(r,\Delta^j f / \Delta^k f)$, $j> k$. Meanwhile, the estimate for $m\left(r,f(z+j) /f(z+k) \right)$ is essentially the same for any $j\not= k$.

The discussion above gives raise to the following difference analogue of Theorem~\ref{thhh1}.
\begin{theorem}\label{thhh2}
Let $\{f_1, \ldots, f_n\}$ be a meromorphic solution base of \eqref{d1} with entire coefficients $A_{0},\ldots,A_{n-1}$ such that at least one of them  is non-constant. Suppose that $p\in\{0,\ldots,n-1\}$ is the smallest index such that \eqref{2} holds. Suppose further that each solution has hyper-order $<1$, and let $\rho_2$ be the maximum of hyper-orders of all the solutions. Then $A_p$ is non-constant, and there are at least $n-p$ solutions $f$ in $\{f_1,\ldots,f_n \}$ with the following property: For any $\veps>0$, there exists a set $I\subset [1,\infty)$ of finite logarithmic measure, such that
   \begin{equation}\label{d3-2}
	T(r,f) \ge  r^{1-\rho_2-\veps} T(r,A_p),\quad r\not\in \big(I\cup[0,1]\big).
	\end{equation}
For these solutions, the value $0$ is the only possible finite deficient value.
\end{theorem}

\begin{remark}
Suppose that $f$ is a meromorphic solution of \eqref{d1} with constant coefficients. 
We make a simple modification of the reasoning in \cite[Section~1]{HKLRT} as follows.
Define $w(z)=f(e^{2\pi iz}+z)$, and denote $\zeta=e^{2\pi iz}+z$. Then 
	$$
	\Delta w(z)=w(z+1)-w(z)=f(\zeta+1)-f(\zeta)=\Delta f(\zeta).
	$$
It follows that $w(z)$ solves \eqref{d1}, and grows much faster than $f(z)$. Further changes
of variable produces a sequence of functions $\{w_n\}$ each solving \eqref{d1} such that
$w_{n+1}$ grows faster than $w_n$ for every $n$. Thus no upper bound for the growth of
solutions of \eqref{d1} can be given in the case of constant coefficients.
\end{remark}




In order to prove Theorem \ref{thhh2}, we need the following difference analogue of Lemma~\ref{5.3}.
\begin{lemma}\label{d5.3}
Suppose that $f_{0,1},\ldots,f_{0,n}$ are linearly independent meromorphic functions of hyper-order $<1$. Define inductively
	\begin{equation}\label{lin-indep-sols2}
	f_{q,s}=\Delta\left(\frac{f_{q-1,s+1}}{f_{q-1,1}}\right),\quad 1\leq q\leq n-1,\ 1\leq s\leq n-q.
	\end{equation}
Then there exists a set $I_1\subset[1,\infty)$ of finite logarithmic measure, such that 
	\begin{equation}\label{d01}
	T(r,f_{q,s})\lesssim \sum_{l=1}^{q+s} T(r,f_{0,l})+1,\quad r\not \in \big(I_1\cup [0,1]\big).
	\end{equation}
\end{lemma}
\begin{proof}
First, suppose that $q=1$. If $1\leq s\leq n-1$, it follows from \eqref{lin-indep-sols2}, \cite[Theorem 5.1]{HKK} and \cite[Lemma~8.3]{HKK} that there exists a set $I_0 \subset[1, \infty)$ of finite logarithmic measure, such that
	\begin{eqnarray*}
	T(r,f_{1,s}(z))&\leq& T\left(r,\frac{f_{0,s+1}}{f_{0,1}}(z+1)\right)+T\left(r,\frac{f_{0,s+1}}{f_{0,1}}(z)\right)+O(1)\\ 
	&\lesssim& T\left(r,\frac{f_{0,s+1}}{f_{0,1}}(z)\right)+N\left(r+1,\frac{f_{0,s+1}}{f_{0,1}}(z)\right) +1 \\
		&\lesssim& T\left(r,\frac{f_{0,s+1}}{f_{0,1}}(z)\right)+1 \\
	&\lesssim& \sum_{l=1}^{1+s} T(r,f_{0,l})+1,\quad r\to\infty, \  r\not \in \big(I_0\cup [0,1]\big).
\end{eqnarray*}
Thus, \eqref{d01} holds for $q=1$.

Second, we assume that \eqref{d01} is true for $q=m-1$, that is, there exists a set $I_0^*\subset[1,\infty)$ of finite logarithmic measure, such that
\begin{equation*}\label{d-induction}
	T(r,f_{m-1,s})\lesssim \sum_{l=1}^{m+s-1} T(r,f_{0,l})+1,\quad r\not \in \big(I_0^*\cup [0,1]\big),
	\end{equation*}
where $1\leq s\leq n-m+1$. Therefore, applying the reasoning from the
case $q=1$ to the functions 
	$$
	f_{m,s}=\Delta\left(\frac{f_{m-1,s+1}}{f_{m-1,1}}\right),\quad 1\leq s\leq n-m,
	$$
the assertion \eqref{d01} follows for $q=m$. This proves \eqref{d01} for $1\leq q\leq n-1$.
\end{proof}

Using the order reduction method for linear difference equations introduced in \cite{RO}, we easily obtain the following difference analogue of Lemma~\ref{coeff_Ap}.

\begin{lemma}\label{d-lemma-or}
Let the coefficients $A_{0},\ldots,A_{n-1}$ in \eqref{d1} be meromorphic functions in $\Bbb C$, and let $f_{0,1},\ldots,f_{0,n}$ be linearly independent 
solutions of the equation \eqref{d1}. Define the functions $f_{q,s}$ as in \eqref{lin-indep-sols2}. Then, for $p \in \{0, 1, \ldots, n-1\}$, we have
	\begin{equation*}\label{d-4}
	-A_{p}= C_{n}+A_{n-1}C_{n-1}+\cdots+A_{p+1}C_{p+1},
	\end{equation*}
where $C_{p+1},\ldots, C_{n}$ have the following form
	\begin{align*}
 \label{d-CC}
	C_j = \sum_{{ l_0 + l_1 + \cdots + l_p = j-p}} &K_{l_0, l_1, \ldots, l_p}
	\,\frac{\Delta^{l_0}f_{0,1}(z+j-l_0)}{f_{0,1}(z+n)}\nonumber\\
	&\,\,\cdots\,\frac{\Delta^{l_{p-1}}f_{p-1,1}(z+j-p+1-l_{p-1})}{f_{p-1,1}(z+n-p+1)}\,\frac{\Delta^{l_p}f_{p,1}(z)}{f_{p,1}(z)}.	
\end{align*}	
Here $0\leq l_0,l_1,\ldots,l_p\leq j-p$ and $K_{l_0, l_1, \ldots, l_p}$ are absolute positive constants.
\end{lemma}

Finally, we need a difference analogue of Wittich's theorem. 

\begin{lemma} \label{W-dif}
Suppose that a meromorphic solution $f$ of \eqref{d1} satisfies \eqref{admissibl}, where $E\subset[1,\infty)$ is a set of finite logarithmic measure. Then $0$ is the only possible finite  Nevanlinna deficient value for $f$.
\end{lemma}
\begin{proof}
Let $a\in\C\setminus \{0\}$. Using the same reasoning as in the proof of \cite[Theorem~4.3]{Laine}, and by using the lemma on the logarithmic differences \cite[Theorem~5.1]{HKK} instead of the lemma on the logarithmic derivatives, we easily obtain
	$$
	\liminf_{r\to\infty} \frac{m\left(r,\frac{1}{f-a} \right)}{T(r,f)}=0.
	$$
This completes the proof.
\end{proof}

\begin{proof}[Proof of Theorem \ref{thhh2}]
We rename the solutions $f_1, \ldots, f_n$ by $f_{0,1}, \ldots, f_{0,n}$. We prove that there are at most $p$ solutions $f$ in $\{f_{0,1}, \ldots, f_{0,n}\}$ satisfying, for some $\veps_0>0$,
	\begin{equation}\label{d3.1-ii}
	\frac{T(r,f)}{r^{1-\rho_2-\varepsilon_0} T(r,A_{p})}  < 1, \quad r\in F,
	\end{equation}
where $F\subset [1,\infty)$ has infinite logarithmic measure. We assume on the contrary to this claim that there are $p+1$ 
solutions $f$ in $\{f_{0,1}, \ldots, f_{0,n}\}$, say $f_{0,1},\ldots,f_{0,{p+1}}$, satisfying \eqref{d3.1-ii}, and aim for a contradiction. 
	
Similarly as in the proof of Theorem~\ref{asymptotic-cor}, we deduce that $A_p$ is non-constant, and therefore $T(r,A_p)$ is unbounded.

From Lemma \ref{d-lemma-or}, we have the immediate estimate
	\begin{equation}\label{d5}
	 m(r,A_{p})\leq \sum_{j=p+1}^n m(r,A_{j})+\sum_{j=p+1}^n m(r,C_j)+O(1).
	\end{equation}
From \cite[Theorem~5.1]{HKK}, for any $\veps \in (0,\veps_0) $, there exists a set $I_2\subset[1,\infty)$ of finite logarithmic measure, such that 
	$$
	\sum_{j=p+1}^n m(r,C_j)\lesssim \sum_{q=0}^{p}\frac{ T(r,f_{q,1})}{r^{1-\rho_2-\veps}}+1, \quad r\notin \big( I_2\cup[0,1]\big).
	$$
Therefore, from Lemma~\ref{d5.3}, it follows that
	\begin{equation}\label{d6}
	\begin{split}
	\sum_{j=p+1}^n m(r,C_j)\lesssim\sum_{l=1}^{p+1}\frac{T(r,f_{0,l})}{r^{1-\rho_2-\veps}} +1,\quad r\notin \big(I_1\cup I_2\cup[0,1]\big).
	\end{split}
	\end{equation}
Analogously as in the proof of Theorem~\ref{asymptotic-cor}, we use \eqref{2}, \eqref{d3.1-ii}, \eqref{d5}
and \eqref{d6} to obtain a contradiction.  Thus, we obtain the conclusion that at least $n-p$ solutions $f$ in $\{f_1, \ldots, f_n\}$ satisfy \eqref{d3-2}. 

From \eqref{2} and from the definition of the index $p$, we easily get
	$$
	T(r,A_j) = O(T(r,A_p)), \quad j\not=p.
	$$
Therefore, combining this with \eqref{d3-2}, we obtain $T(r,A_j) = o(T(r,f))$, $r\to\infty, \, r\notin \big(I\cup[0,1]\big)$.  Hence,  according to Lemma~\ref{W-dif}, the value $0$ is the only possible finite deficient value for the solutions satisfying \eqref{d3-2}. Thus, the theorem is proved.
\end{proof}

\section{Results on linear $q$-difference equations}\label{q-difference-sec}

Consider the $q$-difference equation
\begin{equation}\label{qd1}
	\Delta_q^n f+A_{n-1}\Delta_q^{n-1} f+\cdots+A_{1}\Delta_q f+A_{0}f=0,
	\end{equation}
where $A_0(\not\equiv 0),\ldots,A_{n-1}$ are entire functions, and $\Delta_q$, for $q\in\C\setminus\{0\}$, is a $q$-difference operator defined by $\Delta_q f(z)=f(qz)-f(z)$, $\Delta_q^{n}f(z)=\Delta_q(\Delta_q^{n-1}f(z))$.
Since the results for equations \eqref{qd1} are very similar to those for equations \eqref{d1}, we will state the results in this section without giving the proofs. The reader  will have no problem in verifying the results  by studying the proofs in Section~\ref{difference-sec}.  We begin with a $q$-difference analogue of Theorem~\ref{thhh1}.

\begin{theorem}\label{qthhh2}
Let $\{f_1, \ldots, f_n\}$ be a meromorphic solution base of \eqref{qd1} with entire coefficients $A_{0},\ldots,A_{n-1}$ such that at least one of them  is non-constant. Suppose that $p\in\{0,\ldots,n-1\}$ is the smallest index such that \eqref{2} holds. If each solution is of zero order, then $A_p$ is non-constant, and there are at least $n-p$ solutions $f$ in $\{f_1,\ldots,f_n \}$ satisfying
	\begin{equation}\label{qd3}
	T(r,A_{p})=o( T(r,f)),\quad r\in F,
	\end{equation}
where $F$ is a set of lower logarithmic density 1.
\end{theorem}

Recall that the upper logarithmic density of a set $I\subset [1,\infty)$ is given by
	$$
	\uld(I) = \limsup_{r\to \infty} \frac{1}{\log r}\int_{I\cap [0,r]} \frac{dr}{r}.
	$$
Clearly, $0\le\uld(I)\le 1$. The key lemma on the logarithmic $q$-difference is \cite[Theorem~1.1]{BHMK} and it asserts: \emph{Let $f$ be a non-constant zero-order meromorphic function, and let $q\in\C\setminus  \{0\}$. Then
	$$
	m\left(r, \frac{f(q z)}{f(z)}\right)=o(T(r, f))
	$$
on a set of logarithmic density $1$.}

It is known that if a meromorphic function of zero lower order cannot have more than one deficient value. For that reason, it is not so interesting to study the deficient values of the solutions in Theorem~\ref{qthhh2}. However, by the  definition of the index $p$, the conclusion \eqref{qd3} holds for any $A_j$ in place of $A_p$. This corresponds to the definition of admissible solutions.

Using $T(r,f(qz)) = T(|q|r,f) + O(1)$ in \cite[p.~249]{BIY} and using \cite[Lemma~4]{H2}, we can easily obtain a $q$-difference analogue of Lemma~\ref{5.3}.

\begin{lemma}\label{qd5.3}
Suppose that $f_{0,1},\ldots,f_{0,n}$ are linearly independent meromorphic functions of zero order. Define inductively
	\begin{equation}\label{q-lin-indep-sols2}
	f_{q,s}=\Delta_q\left(\frac{f_{q-1,s+1}}{f_{q-1,1}}\right),\quad 1\leq q\leq n-1,\ 1\leq s\leq n-q.
	\end{equation}
Then
	\begin{equation*}
	T(r,f_{q,s})\lesssim \sum_{l=1}^{q+s} T(r,f_{0,l})+1,\quad r \in F,
	\end{equation*}
where $F\subset[1,\infty)$ is a set of lower logarithmic density 1.
\end{lemma}

By a simple modification of the reasoning in \cite{RO}, we obtain, for a fixed $q\in\{1,\ldots,n-1\}$, that the functions 
    \begin{equation*}
	f_{q,s}=\Delta_q\left(\frac{f_{q-1,s+1}}{f_{q-1,1}}\right),\quad 1\leq q\leq n-1,\ 1\leq s\leq n-q.
	\end{equation*}
are linearly independent solutions of the equation
	\begin{equation*}
	\Delta_q^{n-q} f+A_{q,n-q-1}\Delta_q^{n-q-1} f+\cdots+A_{q,1}\Delta_q f+A_{q,0}f=0,
	\end{equation*}
where
	\begin{equation*}
	A_{q,j}(z) =  \sum_{k=j+1}^{n-q+1} {k\choose j+1} A_{q-1,k}(z)  
	\frac{\Delta_q^{k-j-1}f_{q-1,1}(q^{j+1}z)}{f_{q-1,1}(q^{n-q+1}z)}
	\end{equation*}
holds for $q=1,\ldots,n-1$ and $j=0,\ldots,n-q-1.$ 
Then we follow the proof of Lemma~\ref{coeff_Ap} to obtain a representation for $A_p$ in terms of the coefficients $A_{p+1}, \ldots, A_{n-1}$ and the solution base of \eqref{qd1}.

\begin{lemma}\label{qd-lemma-or}
Let the coefficients $A_{0},\ldots,A_{n-1}$ in \eqref{qd1} be meromorphic functions in $\Bbb C$, and let $f_{0,1},\ldots,f_{0,n}$ be linearly independent 
solutions of the equation \eqref{qd1}. Define the functions $f_{q,s}$ as in \eqref{q-lin-indep-sols2}. Then, for $p \in \{0, 1, \ldots, n-1\}$, we have
	\begin{equation*}\label{qd-4}
	-A_{p}= C_{n}+A_{n-1}C_{n-1}+\cdots+A_{p+1}C_{p+1},
	\end{equation*}
where $C_{j}$, $j= p+1, \ldots, n$, have the following form
	\begin{align*}
 \label{qd-CC}
	C_j = \sum_{{ l_0 + l_1 + \cdots + l_p = j-p}} &K_{l_0, l_1, \ldots, l_p}
	\,\frac{\Delta_q^{l_0}f_{0,1}(q^{j-l_0}z)}{f_{0,1}(q^nz)}\nonumber\\
	&\,\,\cdots\,\frac{\Delta^{l_{p-1}}f_{p-1,1}(q^{j-p+1-l_{p-1}}z)}{f_{p-1,1}(q^{n-p+1}z)}\,\frac{\Delta^{l_p}f_{p,1}(z)}{f_{p,1}(z)}.	
\end{align*}	
Here $0\leq l_0,l_1,\ldots,l_p\leq j-p$ and $K_{l_0, l_1, \ldots, l_p}$ are absolute positive constants.
\end{lemma}

\end{document}